\documentclass[11pt]{article}

 \usepackage{amsmath,amssymb,amscd,amsthm,esint}

\usepackage{graphics,amsthm,mathrsfs,amsfonts}

\usepackage{sidecap}

\usepackage{float}

\usepackage{extarrows}

\usepackage{booktabs}

\usepackage{verbatim}

\usepackage[usenames,dvipsnames]{xcolor}

\usepackage{hyperref}

\newtheorem{thm}{Theorem}[section]
\newtheorem{lemma}[thm]{Lemma}

\newtheorem{prop}[thm]{Proposition}

\theoremstyle{definition}
\newtheorem{remark}[thm]{Remark}
\newtheorem*{defn}{Definition}

\def\XXint#1#2#3{{\setbox0=\hbox{$#1{#2#3}{\int}$}
         \vcenter{\hbox{$#2#3$}}\kern-.5\wd0}}

\def\R{\mathbb{R}}
\def\C{\mathbb{C}}

\def\A{\mathbf{A}}
\def\B{\mathbf{B}}

\def\P{\mathbf{P}}
\def\b{\mathcal{B}}

\numberwithin{equation}{section}

\begin{document}

\title{The Magnetic Laplacian with a Higher-order \\ Vanishing Magnetic Field in a Bounded Domain}

\author{
Zhongwei Shen\thanks{Supported in part by the NSF grant DMS-2153585.}  }
\date{}
\maketitle

\begin{abstract}

This paper is concerned with spectrum properties of the magnetic Laplacian with a higher-order vanishing magnetic field in a bounded domain.
We study the asymptotic behaviors of ground state energies for  the Dirichlet  Laplacian, the Neumann Laplacian, and the Dirichlet-to-Neumann operator, 
as the field strength parameter  $\beta$ goes to infinite.
Assume that the magnetic field   does not vanish to infinite order, 
we establish the leading orders of $\beta$.
We also obtain the first terms in the asymptotic expansions with remainder estimates under additional assumptions on an invariant subspace 
for  a Taylor polynomial of the magnetic field.
Our aim is to provide a unified approach to all three cases.
 
\medskip

\noindent{\it Keywords.  } Schr\"odinger Operator; Magnetic Field; Ground State Energy; Bounded Domain.

\medskip

\noindent {\it MR (2020) Subject Classification}: 35P25.

\end{abstract}


\section{Introduction}\label{section-1}

Consider the magnetic Laplacian,
\begin{equation}\label{op-0}
H(\A)=(D +\A)^2
\end{equation}
in a bounded domain $\Omega$ in $\R^d$, $d\ge 2$,  where $D=-i \nabla$ and  $\A=(A_1, A_2, \dots, A_d): \overline{\Omega} \to \R^d$ is a magnetic potential.
Let  $\lambda^D(\A, \Omega)$, $\lambda^N(\A, \Omega)$, and $\lambda^{D\!N}(\A, \Omega)$ 
denote the ground state energies for the Dirichlet Laplacian, the Neumann Laplacian, and
the Dirichlet-to-Neumann operator, associated with $H(\A)$, defined by
\begin{equation}
\lambda^D (\A, \Omega)
=\inf_{\psi \in C_0^1(\Omega; \C) }
\frac{\int_\Omega |(D +\A)\psi|^2}{\int_\Omega |\psi|^2},
\end{equation}
 \begin{equation}
\lambda^N (\A, \Omega)
=\inf_{\psi \in C^1(\overline{\Omega}; \C ) }
\frac{\int_\Omega |(D +\A)\psi|^2}{\int_\Omega |\psi|^2},
\end{equation}
\begin{equation}
\lambda^{D\!N} (\A, \Omega)
=\inf_{\psi \in C^1(\overline{\Omega}; \C)  }
\frac{\int_\Omega |(D +\A)\psi|^2}{\int_{\partial\Omega} |\psi|^2},
\end{equation}
respectively.
In this paper we are interested in  asymptotic behaviors of $\lambda^D(\beta \A, \Omega)$, $\lambda^N (\beta \A, \Omega)$,  and 
$\lambda^{D\!N}(\beta\A, \Omega)$,  as the field strength parameter $\beta \to \infty$. 
The problems to be studied are equivalent to those  in the semi-classical analysis for the Schr\"odinger operator 
$(h D+\A)^2$ as $h \to 0$.

The spectrum properties  of the magnetic Laplacian  in a bounded domain have been 
investigated  extensively since 1990's.
Let $\B=\nabla \times \A$ denote the magnetic field. 
Much of the existing literature treats  the non-vanishing case, where $\min_{\overline{\Omega}}|\B|>0$, with the Dirichlet or Neumann condition.
The study of the vanishing case, where $\min_{\overline{\Omega}} |\B|=0$,  began in \cite{Montgomery} with the Dirichlet condition 
and has been carried out mostly in the case $d=2$  under the non-degenerate condition: 
$|\nabla B_{12} |\neq 0$ when $B_{12}=0$, where $B_{12}=\partial_1 A_2-\partial_2 A_1$.
In both cases, the first terms in the asymptotic expansions of $\lambda^D (\beta\A, \Omega)$ and $\lambda^N(\beta \A, \Omega)$, with remainder estimates,  
have been obtained  \cite{Montgomery, Helffer-1996, Lu-1999, Helffer-2001, Pan-2002, Miqueu-2018}.
The two-term asymptotic expansions are available under some strict conditions \cite{Helffer-2004, Helffer-2006,  Raymond-2009, Helffer-2023}. 
Also see related work in \cite{Helffer-1997,  Bauman-1998, Pino-2000,  Dauge-2005, Fournais-2007, 
 Raymond-2015, Kachmar-2023, Helffer-2009, Raymond-2013, Raymond-2015a, Attar-2015}.
For a survey and  references  in this area as well as applications to  the theory of  superconductivity, 
we  refer the reader to two expository books, \cite{Helffer-book} by S. Fournais and B. Helffer and \cite {Raymond-book}  by N. Raymond.
 To the best of the author's knowledge, the higher-order vanishing case was only studied in \cite{Helffer-1996, Dauge-2020} (see Remark \ref{re-01} for addtional comments on 
 these two papers).
 
The purpose of this paper is to investigate  the general case,  where $d\ge 2$ and  $|\B|$ may vanish to a higher order
in  $\overline{\Omega}$.
 Inspired by  \cite{CGHP, Helffer-F2024, Helffer-2024}, 
we also study the spectrum of the Dirichlet-to-Neumann operator associated with $H(\beta \A)$,  as $\beta \to \infty$.
We aim to provide a unified approach to the study  of  the asymptotic behaviors of 
$\lambda^D (\beta  \A, \Omega)$, 
$\lambda^N (\beta\A, \Omega)$, and $\lambda^{D\!N} (\beta\A, \Omega)$.
Although the techniques we develop in this paper may be used to establish localization estimates 
for associated  eigenfunctions as in \cite{Helffer-1996}, we will leave them to a future study.

The paper is divided into two parts.
In the first part, we establish  the leading orders of the ground state energies (or the bottoms of the spectra) as
$\beta \to  \infty$. For $\lambda^D(\beta \A, \Omega)$ and $\lambda^N (\beta \A, \Omega)$, 
we assume that $\A$ is smooth and  $|\B|$ does not vanish to infinite order at any point  in $\overline{\Omega}$.
As a result, for each $x\in \overline{\Omega}$, there exists an integer 
$\kappa=\kappa (x)\ge 0$ such that
 \begin{equation}\label{ex-1}
\left\{
\aligned
& \text{$D^\alpha \B(x)=0$ for all $\alpha$ with $|\alpha|\le  \kappa-1$, }\\
&\text{$D^\alpha \B(x) \neq 0$ for some $\alpha$
with $|\alpha|=\kappa$.}
\endaligned
\right.
\end{equation}
Define
\begin{equation}\label{kappa}
\aligned
\kappa_* & =\max \left\{ \kappa (x): x \in \overline{\Omega}\right\} . \\
 \endaligned
\end{equation}
It follows that there exist $C_0, c_0>0$ such that 
\begin{equation}\label{main-c-1}
c_0\le \sum_{|\alpha|\le \kappa_*} 
|\partial^\alpha \B(x)| \quad  \text{ and } \quad
\sum_{|\alpha|\le \kappa_*+1}  |\partial^\alpha \B(x)|  \le C_0
\end{equation}
for any $x\in \overline{\Omega}$.

\begin{thm}\label{main-thm-1}
Let $\Omega$ be a bounded Lipschitz domain in $\R^d, d\ge 2$.
Suppose that  $\A \in C^\infty(\overline{\Omega}; \R^d)$  and  $|\B|$ does not vanish to infinite order  at any point in $\overline{\Omega}$.
Let $\kappa_*\ge 0$ be defined by \eqref{kappa}.
Then
\begin{equation}\label{main-1a}
  c \, \beta^{\frac{2}{\kappa_* +2}}
  \le \lambda^N(\beta\A, \Omega) \le   \lambda^D (\beta \A, \Omega)  \le C  \beta^{\frac{2}{\kappa_*+2}}
 \end{equation}
for $\beta>C$, where $C, c>0$  depend only on $d$, $\kappa_*$, $\Omega$ and $(C_0, c_0)$ in \eqref{main-c-1}.
\end{thm}

\begin{remark}\label{re-01}
{\rm
For  $d=2$ and $\kappa_*=1$, the estimate $\lambda^D(\beta\A, \Omega) \approx \beta^{\frac{2}{3}}$ was proved in \cite{Montgomery}, assuming that 
$\Gamma=\{ x\in \overline{\Omega}: B_{12}(x) =0 \} $ is a smooth curve in $\Omega$ and that  $\nabla B_{12}\neq 0$ on $\Gamma$.
The case of higher-order vanishing was first studied in \cite{Helffer-1996}, where it was proved that
 \begin{equation}\label{pre-4}
 \lambda^D(\beta \A, \Omega)
 \approx \beta^{\frac{2}{\kappa_*+2}}, 
 \end{equation}
 assuming that   $\Gamma=\{ x\in \overline{\Omega}: |\B(x)|=0 \}$ is a submanifold of $\Omega$ or  $\Gamma\subset \Omega$ is discrete and that 
 $|\B (x)| \approx  [\text{\rm dist}(x, \Gamma)]^{\kappa_*} $ for $x\in \Omega$.
 Furthermore, in the case of discrete wells,  a reminder estimate 
 was also established.
 The vanishing case for the Neumann condition was studied in \cite{Pan-2002, Miqueu-2018} for $d=2$ and $\kappa_*=1$, where a one-term asymptotic 
 expansion with a remainder  for $\lambda^N(\beta\A, \Omega)$ was established.
  In \cite{Dauge-2020}, the authors consider the case of second-order vanishing for $d=2$,  assuming  that  $\Sigma=\{ x\in \Gamma: \nabla B_{12} (x)=0 \}$ is finite.
}
\end{remark}

For $\lambda^{D\!N} (\beta\A, \Omega)$, we assume that  $|\B|$ does not vanish to infinite order at any point on $\partial\Omega$.
Define
\begin{equation}\label{kappa-1}
 \kappa_0 =\max \left\{ \kappa (x): x \in \partial\Omega\right\}.
\end{equation}
It follows that there exists $c_1>0$ such that 
\begin{equation}\label{main-c-2}
c_1\le \sum_{|\alpha|\le \kappa_0} 
|\partial^\alpha \B(x)|  
\end{equation}
for any $x\in \partial{\Omega}$.

\begin{thm}\label{main-thm-2}
Let $\Omega$ be a bounded Lipschitz domain in $\R^d, d\ge 2$.
Suppose that  $\A\in C^\infty(\overline{\Omega}; \R^d)$  and $|\B|$ does not vanish to infinite order  at any point on $\partial {\Omega}$.
Let $\kappa_0\ge 0$ be defined by \eqref{kappa-1}.
Then
\begin{equation}\label{main-2a}
  c \, \beta^{\frac{1}{\kappa_0 +2}}
  \le \lambda^{D\!N}(\beta\A, \Omega)  \le C  \beta^{\frac{1}{\kappa_0+2}}
 \end{equation}
for $\beta>C$, where $C, c>0$ depend only on $d$, $\kappa_0$, $\Omega$,   $ \| \B \|_{C^{\kappa_0 +1} (\overline{\Omega})}$ and
$ c_1$ in \eqref{main-c-2}.
\end{thm}

\begin{remark}\label{re-02}
{\rm
The non-vanishing case  $\kappa_0=0$  for $d=2$ was studied recently  in \cite{Helffer-2024}, where  the limit of 
$ \beta^{-\frac12} \lambda^{D\!N} (\beta \A, \Omega) $
as $\beta\to \infty$, was identified.
The paper also obtained a two-term asymptotic expansion for $\lambda^{D\!N}(\beta \A, \Omega)$ in 
the case of constant magnetic fields.
}
\end{remark}

Theorems \ref{main-thm-1} and \ref{main-thm-2}  give the leading orders of  $\beta$ in $\lambda^D(\beta\A, \Omega)$, $\lambda^N(\beta\A, \Omega)$, and
$\lambda^{D\!N} (\beta \A, \Omega)$ for large $\beta$, under minimal assumptions  on $\B$ and $\Omega$.
The proofs for the upper bounds in \eqref{main-1a} and \eqref{main-2a} use quasimodes (test functions) and are  fairly straightforward.
The proofs for the lower bounds  rely on two inequalities,
\begin{equation}\label{in-1}
c\int_\Omega \{ m(x, \B)\}^2 |\psi|^2
\le \int_\Omega |(D+\A)\psi|^2
\end{equation}
and
\begin{equation}\label{in-2}
c\int_{\partial \Omega}   m(x, \B)  |\psi|^2
\le \int_\Omega |(D+\A)\psi|^2
\end{equation}
for $\psi \in C^1(\overline{\Omega}; \C)$, where $m(x, \B)$ defined by \eqref{m}  is a function introduced by the present author  in \cite{Shen-1995}.
To prove \eqref{in-1}-\eqref{in-2}, we use the approach in  \cite{Shen-1998} to connect the
magnetic field with the magnetic potential as well as 
a version of the uncertainty principle by Feffereman and Phong \cite{Fefferman-1983}.
The method of commutators, which involves  integration by parts and has been an important tool  in the study of the magnetic 
Schr\"odinger operators (see \cite{Helffer-book, Raymond-book} for references),  is not used in this paper. 
Our approach avoids the error terms introduced by localization and by the presence of boundaries.
As a result, it works equally well for  all three operators.

In the second part of this paper, 
we derive the formulas for  the leading terms in the asymptotic expansions of
$\lambda^D (\beta\A, \Omega)$,
$\lambda^N(\beta\A, \Omega)$, and
$\lambda^{D\!N} (\beta \A, \Omega)$,   with remainder estimates.
More precisely, if $\Omega$ is a bounded $C^{1, 1}$ domain in $\R^d$, $d\ge 2$, , 
under some general  conditions on $\B$, we show that 
\begin{equation}\label{a-1}
\aligned
\lambda^D (\beta \A, \Omega)
 & =\Theta_D \beta^{\frac{2}{\kappa_* +2}} + O(\beta^{\frac{1}{\kappa_*+2} +\frac{1}{k_*+4}}),\\
\lambda^N (\beta \A, \Omega)
 & =\Theta_N \beta^{\frac{2}{\kappa_* +2}} + O(\beta^{\frac{1}{\kappa_*+2} +\frac{1}{k_*+4}}),\\
 \endaligned
 \end{equation}
 and
 \begin{equation}\label{aa-1}
\Theta_{D\!N} \beta^{\frac{1}{\kappa_0+2}}
-C \beta^{\frac{\kappa_0+3}{(\kappa_0+2)(\kappa_0+4)}} \le   \lambda^{D\!N} (\beta \A, \Omega)
 \le \Theta_{D\!N} \beta^{\frac{1}{\kappa_0 +2}} + C \beta^{\frac{1}{\kappa_0+4}},
\end{equation}
for $\beta>1$, where the coefficients $\theta_D$, $\Theta_N$, and $\Theta_{D\!N}$ are given by \eqref{co-1}-\eqref{co-2} below.
See Theorems \ref{main-thm-g1} and \ref{main-thm-gL1}.
 Our results recover most of the known cases for $\lambda^D(\beta \A, \Omega)$ and $\lambda^N(\beta \A, \Omega)$, including 
\begin{itemize}

\item
 the case of the non-vanishing  fields (see subsection  \ref{sub-nv});
 
 \item
 
 the case of discrete wells (see subsection  \ref{sub-dis});

\item

 the case of first-order vanishing in  two dimensions (see subsection  \ref{sub-order}).

\end{itemize}
  As we indicated earlier, our aim is to provide a unified approach to all three eigenvalue problems.
 We further point out that  our    analysis also yields the upper bounds
 \begin{equation}\label{sup}
 \left\{
 \aligned
 \limsup_{\beta \to \infty}
 \beta^{-\frac{2}{\kappa_*+2}}
 \lambda^D (\beta \A, \Omega)
  & \le \Theta_D,\\
\limsup_{\beta \to \infty}
 \beta^{-\frac{2}{\kappa_*+2}}
 \lambda^N (\beta \A, \Omega)
  & \le \Theta_N,
\endaligned
\right.
\end{equation}
under the assumptions that $\Omega$ is $C^1$ and $|\B|$ does not vanish to infinite order at any point in  $\overline{\Omega}$.
See Theorem \ref{thm-up2a}.
If $\Omega$ is $C^1$ and $|\B|$ does not vanish to infinite order at any point on $\partial\Omega$, we obtain 
\begin{equation}\label{sup-1}
\limsup_{\beta \to \infty} 
\beta^{-\frac{1}{\kappa_0+2}}
 \lambda^{D\!N}  (\beta \A, \Omega)
   \le \Theta_{D\!N}.
  \end{equation}
  See Theorem \ref{thm-up2c}.
  It would be interesting to establish explicit  lower bounds for $\liminf_{\beta \to \infty}$ under the same conditions on $\B$ and $\Omega$.
  
 To describe the constants in \eqref{a-1}-\eqref{aa-1} and \eqref{sup}-\eqref{sup-1}, let 
\begin{equation}\label{gamma}
\aligned
\Gamma_1
&= \left\{ x\in \Omega: \kappa (x)= \kappa_*\right \},\\
\Gamma_2
&= \left\{ x\in \partial\Omega: \kappa (x)= \kappa_*\right \},
\endaligned
\end{equation}
and
\begin{equation}\label{ga-1}
\Gamma_0
= \left\{ x\in  \partial\Omega: \kappa (x)= \kappa_0 \right\},
\end{equation}
where the magnetic field $\B$ vanishes to the maximal orders in $\Omega$ and on $\partial\Omega$, respectively.
For $n\in \mathbb{S}^{d-1} $, let
\begin{equation}\label{hs}
\mathbb{H}_n
=\big\{ x\in \R^d: \  \langle x, n \rangle< 0 \big\}
\end{equation}
denote the half-space  with outward normal $n$.  For each $y \in \R^d$,
let $\A_y$ denote the homogenous (vector-valued) polynomial of degree $\kappa(y)+1$ such that 
$\nabla \times \A_y$ is  the $\kappa(y)^{th}$  (matrix-valued) Taylor polynomial of $\B(x+y)$ at $0$.
The coefficients in \eqref{a-1}-\eqref{aa-1} are given by 
\begin{equation}\label{co-1}
\aligned
\Theta_D & =\min 
\left\{
\inf_{y \in \Gamma_1} \lambda (\A_y, \R^d),
\inf_{y\in \Gamma_2}  \lambda^D(\A_y, \mathbb{H}_{n(y)} ) \right\},  \\
\Theta_N & =\min 
\left\{
\inf_{y \in \Gamma_1} \lambda (\A_y, \R^d),
\inf_{y\in \Gamma_2}  \lambda^N(\A_y, \mathbb{H}_{n(y)})  \right\}, 
\endaligned
\end{equation}
where $\lambda(\A_y, \R^d)=\lambda^D (\A_y, \R^d)$, and
\begin{equation}\label{co-2}
\Theta_{D\!N}
=\inf_{y \in \Gamma_0}
\lambda^{D\!N} (\A_y, \mathbb{H}_{n(y)}),
\end{equation}
where $n(y)$ denotes the outward unit normal to $\partial\Omega$ at $y$.

To prove \eqref{a-1}-\eqref{aa-1}, we fix $y\in \partial\Omega$ and let $\b(y, r)$ denote the ball centered at $y$ with
radius $r$. 
Let 
$$
r=\beta^{-\frac{\kappa+3}{(\kappa+2)(\kappa+4)}},
$$
where  $\kappa =\kappa (y)$ and $\beta$ is large.
The key step is to  establish error estimates for
\begin{equation}\label{est-1}
\aligned
\beta^{-\frac{2}{\kappa+2}} \lambda^D (\beta \A, \b(y, r)\cap \Omega)  & -\lambda^D (\A_{y}, \mathbb{H}_{n(y)} ),\\
\beta^{-\frac{2}{\kappa+2}} \mu^N (\beta \A, \b(y, r)\cap \Omega, \Omega)  & -\lambda^N (\A_{y}, \mathbb{H}_{n(y)} ), 
\endaligned
\end{equation}
and
\begin{equation}\label{est-2}
\beta^{-\frac{1}{\kappa+2}} \mu^{D\!N}  (\beta \A, \b(y, r)\cap \Omega, \Omega) -\lambda^{D\!N}  (\A_{y}, \mathbb{H}_{n(y)} ),
\end{equation}
where 
\begin{equation}\label{mu-1}
\left\{
\aligned
\mu^{N} ( \beta \A, \Omega \cap \b(y, r) ,  \Omega)
& =\inf_{\psi} \frac{\int_{ \b(y, r)\cap \Omega} |(D+\beta \A)\psi|^2}
{\int_{\b(y, r)\cap \Omega} |\psi|^2},\\
\mu^{D\!N} ( \beta \A, \Omega \cap \b(y, r) ,  \Omega)
&=\inf_{\psi} \frac{\int_{ \b(y, r)\cap \Omega} |(D+\beta \A)\psi|^2}
{\int_{\b(y, r)\cap \partial\Omega} |\psi|^2},
\endaligned
\right.
\end{equation}
and the infimums  are  taken over those functions $\psi$  in $C^1(\overline{\Omega\cap \b(y, r)}; \C)$ with $\psi=0$ on
$\Omega\cap \partial \b (y, r)$.
A similar estimate is also needed for the interior case $\b(y, r)\subset \Omega$.
By a perturbation argument, the problem is reduced  to the study of the magnetic Laplacian in $\R^d_+$
with a homogeneous polynomial magnetic potential.
Using  inequalities \eqref{in-1}-\eqref{in-2} obtained in part one and by tiling  the half-space with suitable parallelotopes,
we obtain  bounds for \eqref{est-1}-\eqref{est-2}.
See Theorems \ref{thm-local-1}, \ref{thm-local-2}, \ref{thm-local-3} and \ref{thm-local-4}.

The bounding constants $C$  for  \eqref{est-1}-\eqref{est-2}  depend  crucially on the invariant  subspace $V_y$ for  the 
homogenous polynomial $ \P(x) =\nabla \times \A_{y} (x)$.
This subspace of $\R^d$  is the largest subspace with the property that
$$
\P(x+z) =\P(x) \quad \text{ for any } x\in \R^d \text{ and }  z\in V_y.
$$
It is not hard to show that
\begin{equation}\label{V-0}
V_y =\left\{ z \in \R^d: \langle z, \nabla \partial^\alpha B_{j \ell } (y) \rangle=0
\text{ for any } 1\le j < \ell \le d \text{ and } |\alpha|=\kappa-1 \right\},
\end{equation}
where $\kappa=\kappa(y)$.
As a result, in order to establish the asymptotic expansions in \eqref{a-1}-\eqref{aa-1}
using error estimates for \eqref{est-1}-\eqref{est-2}, 
by some localization and covering arguments, one needs to impose 
some  conditions on $V_y$  to ensure that the estimates  hold uniformly  for $y\in \Gamma_j$, $j=0, 1, 2$.
More precisely, we will assume that  there exists $c>0$ such that
\begin{equation}\label{V-1}
\min_{\substack{v \in \mathbb{S}^{d-1} \cap V_y^\perp}}
\sum_{j, \ell} \sum_{|\alpha|=\kappa_*-1} 
|\langle v, \nabla \partial^\alpha B_{j\ell} (y) \rangle | \ge c
\end{equation}
for any $y \in \Gamma_*=\Gamma_1\cup \Gamma_2$ (for $y\in \Gamma_0$ in the case of \eqref{aa-1}), and that
\begin{equation}\label{V-2}
\max_{z\in V_y} |\langle  z, n(y) \rangle|
\ge c
\end{equation}
for any $y\in \Gamma_2$  (for $y\in \Gamma_0$ in the case of  \eqref{aa-1}). 
We remark that these conditions are satisfied for all existing results on  $\lambda^D(\beta\A, \Omega)$ and
$\lambda^N (\beta \A, \Omega)$ in the vanishing case $\kappa_*= 1$.
Our results on $\lambda^D(\beta \A, \Omega)$ and $\lambda^N(\beta \A, \Omega)$ for $k_*\ge 2$, except the cases considered  in \cite{Helffer-1996, Dauge-2020}, 
as well as those on $\lambda^{D\!N} (\beta\A, \Omega)$ for $\kappa_0\ge 1$ are  new.

Throughout the paper, unless indicated otherwise,
 we assume that $\A\in C^1(\R^d; \R^d)$ and $\Omega$ is a bounded Lipschitz domain.
Thus, there exists $r_0>0$ such that for any $x_0\in \partial\Omega$, 
up to a translation and rotation of the coordinate system, we have
\begin{equation}\label{Lip}
\Omega\cap \b(x_0, r_0)
=\left\{ (x^\prime, x_d) \in \R^d:  \ x_d> \phi(x^\prime) \right\}\cap \b(x_0, r_0),
\end{equation}
where $\phi: \R^{d-1} \to \R$ is a Lipschitz continuous function and  $\b(x_0, r_0)=\{ x\in \R^d: |x-x_0|< r_0 \}$
denotes the ball centered at $x_0$  with radius $r_0$ .
We will use $Q(x_0, r)$ to denote the cube centered at $x_0$ with side length $r$. As usual, 
 $\B=\nabla \times \A$.


\section{Upper bounds, part I}\label{section-up}

Let $\Omega$ be a bounded Lipschitz domain in $\R^d$, $d\ge 2$.
In this section we establish the upper bounds for $\lambda^D(\beta\A, \Omega)$, $\lambda^N(\beta \A, \Omega)$, and
$\lambda^{D\!N}(\beta\A, \Omega)$  in \eqref{main-1a} and \eqref{main-2a} .

\begin{lemma}\label{lemma-u1}
\begin{enumerate}

\item

Assume  $\b (y, r)\subset \Omega$. There  exists 
$\theta\in C^1(\b (y, r); \R)$ such that
\begin{equation}\label{u-1aa}
\left(\fint_{\b (y, r)} | \A+\nabla \theta |^2\right)^{1/2}
\le C r \left(\fint_{\b (y, r) } |\B|^2 \right)^{1/2},
\end{equation}
where $C$ depends only on $d$.

\item

Assume $x_0\in \partial\Omega$. There exist $c_0, C_0>0$, depending on $d$ and $\Omega$, 
with the properties that  if $0< r< c_0 r_0$, there exists
$\theta \in C^1(\b (x_0, r)\cap \Omega; \R)$ such that
\begin{equation}\label{u-1a}
\left(\fint_{\b (x_0, c_0 r)\cap \Omega} | \A+\nabla \theta |^2\right)^{1/2}
\le C r \left(\fint_{\b (x_0, C_0 r) \cap \Omega} |\B|^2 \right)^{1/2},
\end{equation}
where $C$ depends  on $d$ and $\Omega$.

\end{enumerate}

\end{lemma}

\begin{proof}

The estimate \eqref{u-1aa} for the case $\b (y, r) \subset  \Omega$ is well known.
We give a proof of \eqref{u-1a}  for  $x_0\in \partial\Omega$. 
Without the loss of generality, we assume that $x_0=0$ and 
$$
\Omega \cap \b (0, r_0)=
\left\{ (x^\prime, x_d)\in \R^d: x_d >\phi (x^\prime) \right\}\cap \b (0, r_0 ),
$$
where $\phi: \R^{d-1} \to \R$ is a Lipschitz continuous function with  $\phi (0)=0$ and $\|\nabla \phi \|_\infty\le M$.
Observe  that the set
\begin{equation*}
E_r=\{ (x^\prime, x_d):  \ |x^\prime|< r \text{ and } \phi (x^\prime) < x_d < \Lambda  r \}
\end{equation*}
is star-shape with respect to any point in $\b (y_0, \delta r)$, where $y_0=(0, \Lambda/2)$, if $\Lambda=\Lambda(d, M)>1$ is sufficiently large and $\delta=\delta(d, M)>0$ is sufficiently small.
Choose $c_0, C_0>0$, depending on $d$ and $M$,  so  that 
\begin{equation}\label{u-1b}
\b (0, c_0 r)\cap \Omega  \subset E_r \subset \b (0, C_0 r) \cap \Omega\subset \b (0, r_0) \cap \Omega
\end{equation}
 for $0< r< c_0 r_0$.
Let  $\widetilde{\A}=(\widetilde{A}_1, \widetilde{A}_2, \dots, \widetilde{A}_d)$, where
\begin{equation}
\widetilde{A}_j (x)
=- \fint_{\b (y_0,\delta r)}
\left\{
\sum_{k=1}^d (x_k-y_k)
\int_0^1 B_{jk} (y + t(x-y)) t \, dt \right\} dy, 
\end{equation}
 and $B_{jk }=\partial_j A_k- \partial_k A_j$.
 A computation shows that 
$
\widetilde{\A}= \A +\nabla \theta
$
in $E_r$, where
$$
\theta (x)=-\fint_{B(y_0, \delta r)}
\left\{ \sum_{k=1}^d (x_k-y_k)
\int_0^1 A_k (y + t(x-y)) t \, dt \right\} dy.
$$
Note that for $x\in E_r$,
$$
\aligned
|\widetilde{\A}(x)|
&\le \fint_{B(y_0, \delta r)}
\left\{
\int_0^{|y-x|} |\B \Big(x + t \frac{y-x}{|y-x|}\Big) |dt \right\} dy\\
&\le C \int_{\mathbb{S}^{d-1}} d\omega
\int_0^\infty W (x+ t\omega) dt\\
&=C \int_{\b (0, Cr) } \frac{W (x+ y)}{|y|^{d-1}}\, dy,
\endaligned
$$
where $W=|\B|\chi_{E_r}$. It follows by Minkowski's inequality for integrals that
$$
\left(\int_{E_r} |\widetilde{\A}|^2 \right)^{1/2}
\le C r \left(\int_{E_r} |\B|^2 \right)^{1/2}.
$$
In view of \eqref{u-1b}, this gives \eqref{u-1a}.
\end{proof}

\begin{thm}\label{thm-u1}
Suppose that there exist $x_0\in \overline{\Omega}$ and $\kappa\ge 0$ such that
\begin{equation}\label{u2-0}
\left(\fint_{\b (x_0, r)\cap \Omega} |\B|^2 \right)^{1/2}
\le C_0 r^\kappa
\end{equation}
for any $0< r< r_0$. Then
\begin{equation}\label{u2-1}
\lambda^N(\beta \A, \Omega) \le 
\lambda^D(\beta \A, \Omega) \le C \beta^{\frac{2}{\kappa+2}}
\end{equation}
for any $\beta>  1$, where $C$ depends only on $d$, $\Omega$, and $(C_0, \kappa)$.
\end{thm}

\begin{proof}

Let $\b (y, r)\subset \Omega$.
By Lemma \ref{lemma-u1}, there exists $\theta  \in C^1(\b (y, r); \R^d)$ such that \eqref{u-1aa} holds.
Let $\psi \in C_0^1 (\b (y, r); \R)$ to be determined.
Note that
$$
(D+\beta \A)(\psi e^{i\beta  \theta})
= e^{i\beta  \theta} (D+\beta \A+\beta \nabla \theta)\psi.
$$
It follows that
$$
|(D+\beta \A)(\psi e^{i \beta \theta})|^2
= |\nabla \psi|^2 +\beta^2  |\A +\nabla \theta|^2 |\psi|^2.
$$
We now choose $\psi \in C_0^1(\b (y, r); \R)$ such that
$0\le \psi \le 1$, $\psi=1$ in $\b (y, r/2)$ and $|\nabla \psi |\le C /r$.
This gives
\begin{equation}\label{u2-3}
\aligned
\lambda^D (\beta \A, \Omega)
&\le
\frac{\int_\Omega |(D+\beta \A)(\psi e^{i \beta \theta} )|^2}
{\int_\Omega |\psi  e^{i\beta \theta}|^2} \\
& \le C \left\{
\frac{1}{r^2}
+ \beta^2 \fint_{\b (y, r)} |\A +\nabla \theta |^2 \right\}\\
&\le C \left\{ \frac{1}{r^2}
+ \beta^2 r^2 
\fint_{\b (y, r)} |\B|^2 \right\},
\endaligned
\end{equation}
where we have used \eqref{u-1aa} for the last inequality.
Now,  for  $\beta>  1$, let  $r_1= c\,  \beta^{-\frac{1}{\kappa+2}}< r_0$.
Choose $y\in \Omega$ so that $\b (y, cr_1) \subset \b (x_0, r_1)\cap \Omega$.
It follows from  \eqref{u2-0} that
$$
\left(\fint_{\b (y, cr_1)} |\B|^2 \right)^{1/2}
\le C r_1^\kappa,
$$
where $C$ depends on $\Omega$ and $(C_0, \kappa)$. This, together with \eqref{u2-3} with $r=cr_1$, yields
$$
\aligned
\lambda^D(\beta \A, \Omega)
 & \le C \left\{ \frac{1}{r_1^2}
 + \beta^2  r_1^{2+2\kappa} \right\}\\
 &= C \beta^{\frac{2}{\kappa+2}},
 \endaligned
$$
where $C$ depends only on $d$, $\Omega$, and $(C_0, \kappa)$.
\end{proof}

\begin{thm}\label{thm-u4}
Suppose that there exist $x_0\in \partial{\Omega}$ and $\kappa\ge 0$ such that
\begin{equation}\label{u4-0}
\left(\fint_{\b (x_0, r)\cap \Omega} |\B|^2 \right)^{1/2}
\le C_0 r^\kappa
\end{equation}
for any $0< r< r_0$. Then
\begin{equation}\label{u4-1}
\lambda^{D\!N}(\beta \A, \Omega) \le C \beta^{\frac{1}{\kappa+2}}
\end{equation}
for any $\beta> 1$, where $C$ depends only on $d$, $\Omega$,  and $(C_0, \kappa)$.
\end{thm}

\begin{proof}

Let  $0< r< c_0r_0$, where $c_0>0$ is given 
by Lemma \ref{lemma-u1}. There exists $\theta \in C^1 (\b (x_0, c_0 r)\cap \Omega; \R)$ such that 
\eqref{u-1a} holds.
Choose  $\psi \in C_0^1(\b (x_0, c_0 r); \R)$  such that $0\le \psi \le 1$,
$\psi=1$ in $\b (x_0, c_0 r/2)$ and $|\nabla \psi|\le C/r$.
As in the proof of Theorem \ref{thm-u1}, 
\begin{equation}\label{u4-3}
\aligned
\lambda^{D\!N} (\beta \A, \Omega)
&\le
\frac{\int_\Omega |(D+\beta \A)(\psi e^{i \beta \theta} )|^2}
{\int_{\partial\Omega} |\psi  e^{i\beta \theta}|^2} \\
& \le C \left\{
\frac{1}{r}
+ \beta^2  r \fint_{\b (x_0, c_0 r)\cap \Omega} |\A +\nabla \phi|^2 \right\}\\
&\le C \left\{ \frac{1}{r}
+ \beta^2 r^3 
\fint_{\b (x_0, C_0r)\cap \Omega} |\B|^2 \right\}.
\endaligned
\end{equation}
For $\beta > 1$, let $r=c\,  \beta^{-\frac{1}{2+\kappa}}< c_0 r_0$.
It follows from \eqref{u4-3} that
$$
\aligned
\lambda^{D\!N}(\beta \A, \Omega)
 & \le C \left\{ \frac{1}{r} + \beta^2 r^{3+2\kappa}\right\}\\
& = C \beta^{-\frac{1}{\kappa+2}},
\endaligned
$$
where $C$ depends on $d$, $\Omega$ and $(C_0, \kappa)$.
\end{proof}

\begin{remark}\label{re-up}
{\rm
Let $\kappa_*$ be defined by \eqref{kappa}.
Choose  $x_0\in \overline{\Omega}$ such that $\kappa (x_0)=\kappa_*$.
It follows that $\partial^\alpha \B(x_0)=0$ for any $\alpha$ with $|\alpha|\le \kappa_*-1$.
Then, by Taylor's Theorem,
$$
|\B(x)|\le C | x-x_0|^{\kappa_*} \qquad 
\text{ for any  $x\in \b(x_0, r_0)\cap \Omega$},
$$
 where $C$ depends on $d$, $\kappa_*$, $\Omega$ and $\| \B\|_{C^{\kappa}(\overline{\Omega}) }$.
In view of  Theorem \ref{thm-u1}, this gives the upper bound in \eqref{main-1a}.
Similarly, the upper bound in \eqref{main-2a} follows readily  from Theorem \ref{thm-u4}.
}
\end{remark}


\section{Operator lower bounds}\label{section-op}

Throughout this section we assume that $\A\in C^1(\R^d; \R^d)$ and
$\Omega$ is a bounded Lipschitz domain in $\R^d$.
Let 
$$
\widetilde{\Omega}=\{ x\in \R^d: \text{\rm dist}(x, \Omega)< r_0 \}.
$$
We will   prove the inequalities \eqref{in-1}-\eqref{in-2} under  the  additional condition:
\begin{itemize}

\item

 There exists an increasing continuous function $\eta$ on $[0, 2]$ with $\eta (0)=0$ such that
 if $\b(x_0, r)\subset \widetilde{\Omega}$, then 
 \begin{equation}\label{c-1} 
|\B(x)-\B(y)|\le \eta \left( \frac{|x-y|}{r}\right) \fint_{\b(x_0, r)} |\B|
\end{equation}
for any  $x, y\in \b (x_0, r)$.

\end{itemize}

A few remarks about the condition \eqref{c-1} are in order.

\begin{remark}\label{re-op-0}
{\rm
It follows from \eqref{c-1}  that 
\begin{equation}\label{A}
\max_{\b(x_0, r)} |\B|
\le C_0 \fint_{\b(x_0, r)} |\B|
\end{equation}
for any $\b(x_0, r) \subset \widetilde{\Omega}$, where $C_0=1+ \eta (2)$.
This  implies that the function $|\B|$ is an $A_\infty$ weight in $\widetilde{\Omega}$.
 Consequently,   there exist
positive constants $C, c, \delta_0, \delta_1$, depending only on $d$ and $C_0$ in \eqref{A},  such that  
\begin{equation}\label{A-inf}
c \left( \frac{|E|}{|\b (x_0, r) |} \right)^{\delta_0}
\le \frac{\int_E |\B|}{\int_{\b (x_0, r) } |\B|}
\le C  \left( \frac{|E|}{|\b (x_0, r) |} \right)^{\delta_1},
\end{equation}
whenever $\b(x_0, r) \subset \widetilde{\Omega}$ and $E\subset \b (x_0, r) $ is measurable \cite{Coifman}.
In particular, $|\B|$ satisfies the doubling condition, 
\begin{equation}\label{cond-d}
\int_{\b (x_0, 2r)} |\B| \le C \int_{\b (x_0, r)} |\B|
\end{equation}
for any $\b(x_0, 2r) \subset \widetilde{\Omega}$.
As a result, the inequalities  \eqref{c-1}-\eqref{cond-d}
continue to hold  (with different constants and $\eta$), if we replace $\b (x_0, r)$ with the cube $Q(x_0, r)$.
}
\end{remark}

\begin{remark}\label{re-op-1}
{\rm
Suppose $\B$ is a (matrix-valued) polynomial of degree $\kappa$ in $\R^d$.
Then
\begin{equation}\label{c-1a}
\max_{\b(x_0, r)} |\nabla \B|\le \frac{C}{r} \fint_{\b(x_0, r)} |\B|
\end{equation}
for any $\b(x_0, r)\subset \R^d$, where $C$ depends only on $d$ and $\kappa$.
It follows that the condition \eqref{c-1} holds with $\eta (t)= Ct$.
}
\end{remark}

\begin{remark}\label{re-op-2}
{\rm
Suppose $\B\in C^{\kappa+1}(\widetilde{\Omega}; \mathbb{R}^{d\times d})$.
Assume there exist $C_0, c_0>0$ such that 
\begin{equation}\label{c-1b}
c_0\le \sum_{|\alpha|\le \kappa} |\partial^\alpha \B(x)|
\quad \text{ and } \quad
\sum_{|\alpha|\le \kappa+1} |\partial^\alpha \B(x)|\le C_0
\end{equation}
for any $x\in \widetilde{\Omega}$.
Then the inequality \eqref{c-1a} holds  for any $\b(x_0, r)\subset \widetilde{\Omega}$. As a result, 
the condition \eqref{c-1b} implies \eqref{c-1}, with $\eta (t)=C t$, where  $C$ depending on $d$, $\Omega$,  and $(\kappa, C_0, c_0)$.

To see this, let $Q(x_0, r)\subset \widetilde{\Omega}$.
Let $\P$ denote the $\kappa^{th}$ Taylor polynomial of $\B$ at $x_0$. Note that 
$$
\max_{Q(x_0, r)} |\nabla \P|
\le \frac{C}{r} \fint_{Q(x_0, r)} |\P| 
$$
and
$$
\max_{Q(x_0, r)} |\P|  \le C \fint_{Q(x_0, r)} |\P|
\approx \sum_{|\alpha|\le \kappa} 
|\partial^\alpha \B(x_0)| r^{|\alpha|},
$$
where $C$ depends only on $d$ and  $\kappa$.
It follows that  for $0< r< 1$,
$$
\aligned
c_0\,  r^\kappa  & \le \sum_{|\alpha|\le \kappa} |\partial^\alpha \B(x_0)| r^{|\alpha|}
  \le C \fint_{Q(x_0, r)} |\P|\\
 & \le C \fint_{Q(x_0, r)} |\B| + C r^{\kappa+1},
\endaligned
$$
where we have used \eqref{c-1b}.
Thus,  for  $0<r<c$, where $c$ is sufficiently small,
$$
r^\kappa \le C  \fint_{Q(x_0, r)} |\B |.
$$
Furthermore, if $Q(x_0, r)\subset \widetilde{\Omega}$ and $0< r< c$, 
$$
\aligned
\max_{Q(x_0, r)} |\nabla \B|
 & \le \max_{Q(x_0, r)}  |\nabla \P| + C r^{\kappa}
 \le \frac{C}{r} \fint_{Q(x_0, r)} |\P| + C r^\kappa\\
 &\le \frac{C}{r} \fint_{Q(x_0, r)} |\B| + C r^\kappa
 \le \frac{C}{r} \fint_{Q(x_0, r)} |\B|. 
 \endaligned
 $$
 By dividing cubes into subcubes, one may eliminate  the condition $0< r<c$.
 Using the doubling condition, one may also replace $Q(x_0, r)$ by $\b(x_0, r)$.
}
\end{remark}

\begin{lemma}\label{lemma-op1}
 Let $x_0 \in \overline{\Omega}$ and $0< r< r_0$,
where either $x_0\in \partial\Omega$ or $\b (x_0, r) \subset  \Omega$.
Assume  that  for some  measurable set $E\subset \b (x_0, r) \cap \Omega$ and  some $\psi \in C^1(\b (x_0, r) \cap\Omega; \C)$, 
\begin{equation}\label{op1-1}
\frac{\alpha_0}{r^2}
\int_E |\psi|^2
\le \int_{\b (x_0, r) \cap \Omega} |(D+\A)\psi|^2,
\end{equation}
where $\alpha_0\in (0, 1)$.
Then
\begin{equation}\label{op1-2}
\frac{ c\,  \alpha_0 |E|}{ r^{d+2}}
\int_{\b (x_0, r)\cap\Omega} |\psi|^2
\le \int_{\b (x_0, r) \cap \Omega} |(D+\A)\psi|^2,
\end{equation}
where $c>0$ depends only on $d$ if $\b (x_0, r)\subset  \Omega$,  and   $c$ also depends on  the Lipschitz character of $\Omega$ if $x_0\in \partial\Omega$.
\end{lemma}

\begin{proof}

We give the proof  for the case $x_0\in \partial\Omega$. The other case, where $\b (x_0, r) \subset \Omega$, is similar.
By Poincar\'e inequality,
\begin{equation}\label{op1-3}
\aligned
\frac{c}{r^{2}}
\fint_{\b (x_0, r) \cap \Omega}
\int_{\b (x_0, r) \cap \Omega}
\big| |\psi (x)| -|\psi (y)|\big|^2\, dx dy
 & \le \int_{\b (x_0, r)\cap \Omega} 
\big|\nabla |\psi |\big |^2\\
& \le \int_{\b (x_0, r) \cap \Omega} 
|(D+\A)\psi |^2, 
\endaligned
\end{equation}
where we have used the diamagnetic  inequality $\big| \nabla |\psi | \big|\le | (D+\A)\psi|$
for the last step. The constant $c$ in \eqref{op1-3} depends only on $d$ and the Lipschitz character of $\Omega$.
Let $V(x)=\alpha_0 r^{-2} \chi_E (x)$.
Note that  by the assumption \eqref{op1-1}, 
\begin{equation}\label{op1-4}
t \fint_{\b (x_0, r) \cap \Omega}
\int_{\b(x_0, r) \cap \Omega}
 V(x) |\psi (x)|^2 \, dx dy
  \le  t \int_{\b (x_0, r) \cap \Omega} 
|(D+\A)\psi |^2
\end{equation}
for any $t>0$.
It follows by adding  \eqref{op1-3} and \eqref{op1-4} that 
\begin{equation}\label{op1-5}
\aligned
 & \fint_{\b (x_0, r) \cap \Omega}\int_{\b (x_0, r) \cap \Omega}
\left\{
\frac{c}{r^2} \big| |\psi (x)| -|\psi (y)| \big|^2
+ t V(x) |\psi (x)|^2 \right\} dx dy\\
&\qquad  \le (1+t) 
  \int_{\b (x_0, r) \cap \Omega} 
|(D+\A)\psi |^2.
\endaligned
\end{equation}
Choose $t$ so that $\alpha_0 t=c$.  Using 
$$
\aligned
\frac{c}{r^2} \big| |\psi (x)| -|\psi (y)| \big|^2
+ t V(x) |\psi (x)|^2
 & \ge \min \left\{ \frac{c}{r^2}, t V(x) \right\}
\left\{ \big | |\psi (x)| - |\psi (y)| \big|^2
+ |\psi(x)|^2  \right\}\\
& \ge \frac12  \min \left\{ \frac{c}{r^2}, t V(x) \right\}
|\psi (y)|^2\\
& = \frac{c}{2r^2} \chi_ E (x) |\psi (y)|^2, 
\endaligned
$$
we deduce from \eqref{op1-5} that
\begin{equation}\label{op1-6}
\frac{c |E|}{2r^2 |\b (x_0, r) \cap \Omega|}
\int_{\b (x_0, r)\cap \Omega} |\psi|^2
\le 
(1+t) 
  \int_{\b (x_0, r) \cap \Omega} 
|(D+\A)\psi |^2, 
\end{equation}
where $t=c \alpha^{-1}_0$.
Since $x_0 \in \partial\Omega $ and $\Omega$ is Lipschitz, we have $|\b (x_0, r) \cap \Omega| \ge c\,  r^d$.
As a result, \eqref{op1-2} follows readily  from \eqref{op1-6}.
\end{proof}

Recall that $B_{jk}=\partial_j A_k -\partial_k A_j$ for $1\le j, k \le d$.

\begin{lemma}\label{lemma-op2}
Let $d=2$ and $Q$ be a square in $\R^2$. Suppose that
\begin{equation}\label{op2-1}
|Q|\left(\fint_Q |B_{12}|^2 \right)^{1/2}\le \pi.
\end{equation}
Then
\begin{equation}\label{op2-2}
\frac{c}{|Q|}
\left| \int_Q B_{12} \right|^4
\int_Q |\psi|^2 \le  \int_Q | (D+\A) \psi |^2
\end{equation}
for any $\psi\in C^1(Q; \C)$, where $c>0$ is an absolute constant.
\end{lemma}

\begin{proof}

See \cite[Theorem 3.2 with $p=2$]{Shen-1998}.
\end{proof}

\begin{lemma}\label{lemma-op3}
Assume   $\B$ satisfies \eqref{c-1}.
Let $Q\subset \Omega$ be a cube with side length $r< r_0$. Suppose that 
\begin{equation}\label{op3-1}
\frac{\alpha_0}{r^2}
\le \max_{Q} |\B|  \le \frac{1}{r^2}
\end{equation}
for some $\alpha_0\in (0, 1)$. Then, for any $\psi \in C^1(Q; \C)$, 
\begin{equation} \label{op3-2}
\frac{c }{r^2}
\int_{Q}  |\psi |^2 
\le \int_{Q}  |(D+\A)\psi|^2,
\end{equation}
where $c>0$ depends only on $d$, $\alpha_0$  and $\eta$ in \eqref{c-1}.
\end{lemma}

\begin{proof}

It follows from \eqref{op3-1}, \eqref{A} and the doubling condition \eqref{cond-d} that
\begin{equation}\label{op3-3}
\fint_Q |\B| \ge \frac{c}{r^2}, 
\end{equation}
where $c>0$ depends on $d$, $\alpha_0$ and $\eta$. This implies that 
there exist $1\le j < k  \le d$ such that 
\begin{equation}\label{op3-4}
\fint_Q |B_{jk} | \ge \frac {c }{r^2},
\end{equation}
where $c>0$ depends on $d$, $\alpha_0$ and $\eta$.
Without the  loss of generality, we may assume $j=1$ and $k=2$.
Write $Q=I_1\times I_2\times \cdots \times I_d$, where $\{I_m\}$ are open intervals of $\R$ with length $r$.
By Fubini's Theorem, there exists some $x^\prime \in I_3 \times\cdots \times  I_d$ such that 
\begin{equation}\label{op3-5}
\fint_{I_1 \times I_2} |B_{12} (x_1, x_2, x^\prime)|\, dx_1 dx_2 \ge \frac{c}{r^2}.
\end{equation}
Let $g(x_1, x_2)=B_{12}(x_1, x_2, x^\prime)$.
For each  $\gamma=2^{-\ell}$, 
by a partition of  $I_1\times I_2$ into dyadic sub-squares,  it is not hard to see that there exists a square $S$ in $I_1\times I_2$ with side length $\gamma r$ such that  
\begin{equation}\label{op3-6}
\fint_S |g| \ge c\,   r^{-2} .
\end{equation}
Note that by \eqref{c-1} and \eqref{op3-1}, 
\begin{equation}\label{op3-6a}
|\B(x)-\B(y)|\le C r^{-2} \eta \left(\frac{c|x-y|}{r} \right)
\end{equation}
for any $x, y\in Q$. It follows that 
$$
\fint_S \big|g -\fint_S  g \big|
\le  C \eta (C\gamma)   r^{-2}.
$$
Thus, if we choose $\gamma>0 $ so small  that $C \eta (C\gamma)  \le c/2$, then
$$
\aligned
\big| \fint_S g \big|
& \ge \fint_S |g|
- \fint_S \big| g -\fint_S g \big|\\
& \ge  c  r^{-2} - C  \eta (C \gamma)   r^{-2}
\ge (c/2)  r^{-2}.
\endaligned
$$
As a result, we obtain 
$$
\big| \int_S g \big| \ge c\gamma^2/2.
$$
In view of Lemma \ref{lemma-op2}, this leads to 
\begin{equation}\label{op3-7}
\frac{c}{r^2}
\int_{I_1 \times I_2} |\psi |^2\, dx_1 dx_2
\le 
\int_{I_1 \times I_2}
\left\{
|(D_1+A_1)\psi|^2 + |(D_2 +A_2) \psi |^2 \right\}\, dx_1 dx_2,
\end{equation}
where $A_1=A_1(x_1, x_2, x^\prime)$ and $A_2=A_2(x_1, x_2, x^\prime)$.

Next, we observe that  by \eqref{op3-5} and \eqref{op3-6a},
$$
\int_{I_1\times I_2} |B_{12}(x_1, x_2, y^\prime)|\, dx_1 dx_2 \ge c/2, 
$$
if $y^\prime \in I_3 \times \cdots \times I_d $ and $|y^\prime -x^\prime| < \delta r$, where $\delta >0$ 
is sufficiently small.
Hence, the inequality \eqref{op3-7} continues to  hold  with
$A_1=A_1(x_1, x_2, y^\prime)$ and $A_2=A_2 (x_1, x_2, y^\prime)$.
By integrating in $y^\prime$, we see that
\begin{equation}\label{op3-8}
\frac{c}{r^2}
\int_E
|\psi |^2 \, dx
\le \int_{Q} | (D+\A)\psi|^2,
\end{equation}
where $E=I_1 \times I_2 \times \{ y^\prime\in I_3 \times\cdots \times  I_d: |y^\prime -x^\prime|< \delta r\}$.
Finally, note that $|E|\ge c\,  r^d$. By Lemma \ref{lemma-op1} and \eqref{op3-8},  we obtain \eqref{op3-2}. 
\end{proof}

For $x\in \R^d$, define $m(x, \B)$ by
\begin{equation}\label{m}
\frac{1}{m(x, \B)}
=\sup \left\{ r>0:
\max_{Q(x, r)} |\B| \le \frac{1}{r^2} \right\}.
\end{equation}

\begin{lemma}\label{lemma-op3a}
Suppose $ \B$ satisfies the condition \eqref{A} in $ \widetilde{\Omega}$.
Also assume that 
\begin{equation}\label{op3a-1}
\max_{Q(x, r_0) }  |\B|> r_0^{-2} \quad \text{ for any $x\in \overline{\Omega} $.}
\end{equation}
Then, if  $x, y\in \overline{\Omega}$ and $|x-y|<  \frac{1}{m(x, \B)}$,
\begin{equation}\label{m-2}
m(x, \B) \le C m(y, \B)  \quad \text{ and }\quad
m(y, \B) \le C m( x, \B),
\end{equation}
where $C$ depends on $d$ and the constant in \eqref{A}.
\end{lemma}

\begin{proof}
Let $r=\{ m(x, \B)\}^{-1}$.
The condition \eqref{op3a-1} implies  $r< r_0$.
By \eqref{A-inf} and \eqref{cond-d}, if  $|x-y|< r$, then
$$
\max_{Q(y, r)} |\B|
\approx \max_{Q(x, r) } |\B| = \frac{1}{r^2},
$$
from which the estimates in \eqref{m-2} follow by the definition  \eqref{m}.
\end{proof}

\begin{thm}\label{thm-op1}
Suppose $\B$ satisfies the condition \eqref{c-1} in $\widetilde{\Omega}$.
 Also assume \eqref{op3a-1} holds.
Then for any $\psi \in C^1(\overline{\Omega}; \C)$, 
\begin{equation}\label{op4-1}
 c \int_\Omega \{ m(x, \B)\}^2 |\psi |^2
\le \int_\Omega |(D+\A)\psi |^2,
\end{equation}
 where $c>0$ depends only on $d$, $\eta$ in \eqref{c-1},  and the Lipschitz character of $\Omega$.
\end{thm}

\begin{proof}

Let $x_0\in \overline{\Omega}$ and $r= c \{ m(x_0, \B\}^{-1}\le r_0/4$. We will show that 
\begin{equation}\label{op4-2}
\frac{1}{r^2} \int_{\b (x_0, r)\cap \Omega}    |\psi |^2\, dx
\le C \int_{\b (x_0, 3r)\cap \Omega} |(D+\A)\psi|^2\, dx.
\end{equation}
Assume \eqref{op4-2} for a moment. Then, by \eqref{m-2}, 
\begin{equation}\label{op4-2a}
 \int_{\b (x_0, r)\cap \Omega}   \{ m(x, \B)\}^{2-d}   |\psi |^2\, dx
\le C \int_{\b (x_0, 3r)\cap \Omega} \{ m(x, \B) \}^{-d}  |(D+\A)\psi|^2\, dx.
\end{equation}
By integrating both sides of \eqref{op4-2a} in $x_0$ over the domain $\Omega$ and applying Fubini's Theorem, we obtain 
\begin{equation}\label{op4-3}
\aligned
& \int_\Omega \{ m(x, \B)\}^{2-d} |\psi (x)|^2  |\big\{ x_0 \in \Omega: |x_0-x|< c \{ m(x_0, \B)\}^{-1}\big\}| \, dx\\
&\qquad
\le C \int_\Omega \{ m(x, \B)\}^{-d}  |(D+\A)\psi |^2 
|\big\{ x_0\in \Omega: |x_0-x|< 3 c \{ m(x_0, \B) \}^{-1}\big\} |\, dx.
\endaligned
\end{equation}
Note that by Lemma \ref{lemma-op3a}, 
$$
\aligned
 |\big\{ x_0 \in \Omega: |x_0-x|< c \{ m(x_0, \B)\}^{-1}\big\}|
&\ge c \{ m(x, \B)\}^d,  \\
 |\big\{ x_0 \in \Omega: |x_0-x|< 3c \{ m(x_0, \B)\}^{-1}\big\}|
&\le C \{ m(x, \B)\}^d\\
\endaligned
$$
for $x\in \Omega$. The inequality \eqref{op4-1} follows readily from \eqref{op4-3}.

To prove \eqref{op4-2}, we consider three cases.

Case (1). Suppose $\b (x_0, r)\subset \Omega$. Let $Q(x_0, c_d r) \subset B(x_0, r)$.
Since 
$$
\frac{c}{r^2}\le \max_{Q(x_0,c_d r)} |\B|\le \frac{1}{r^2}, 
$$
it follows by  Lemma \ref{lemma-op3} that 
$$
\frac{c}{r^2} \int_{Q(x_0, c_dr)} |\psi|^2
\le \int_{Q(x_0, c_d r)} |(D+\A)\psi|^2.
$$
In view of Lemma \ref{lemma-op1}, this  yields
$$
\frac{c}{r^2} \int_{\b (x_0, r)} |\psi|^2
\le \int_{\b (x_0, r)} |(D+\A)\psi|^2.
$$

Case (2). Suppose $x_0 \in \partial\Omega$.
Since $\Omega$ is a Lipschitz domain, there exists $y_0\in \R^d$ such that 
$Q(y_0, cr) \subset Q(x_0, r) \cap \Omega$, where $c$ depends on $d$ and the Lipschitz character of $\Omega$.
Note that by \eqref{A-inf},
$$
\frac{c}{r^2} \le \max_{Q(y_0, c r)} |\B | \le \frac{1}{r^2}.
$$
It follows by Lemma \ref{lemma-op3} that 
$$
\frac{c}{r^2} \int_{Q(y_0, c  r)}  |\psi|^2 
\le \int_{Q(y_0, c r)} |(D+\A ) \psi|^2.
$$
Again, by Lemma \ref{lemma-op1}, this  gives 
$$
\frac{c}{r^2} \int_{\b(x_0, r)\cap \Omega} |\psi^2
\le \int_{\b (x_0,  r)\cap \Omega} |(D+\A)\psi|^2.
$$

Case (3). Suppose $x_0 \in \Omega$ and $\b(x_0, r) \cap \partial\Omega \neq \emptyset$.
Let $y_0 \in \b (x_0, r) \cap \partial\Omega$.
Observe that  $\b (x_0, r) \subset \b (y_0, 2r) \subset \b (x_0, 3r)$.
It follows that
\begin{equation}\label{op4-8}
\aligned
\frac{1}{r^2} \int_{\b (x_0, r)\cap \Omega} |\psi|^2
&\le \frac{1}{r^2}
\int_{\b (y_0, 2r)\cap \Omega} |\psi |^2
 \le C \int_{\b (y_0, 2r) \cap \Omega} |(D+\A)\psi |^2\\
& \le C \int_{\b (x_0, 3r)\cap \Omega} |(D+\A)\psi|^2.
\endaligned
\end{equation}
We  point out  that the second inequality in \eqref{op4-8} follows from the proof  of case (2) as well as the fact that
$r\approx \{ m(y_0, \B)\}^{-1}$. 
This completes the proof of \eqref{op4-2}. 
\end{proof}

\begin{remark}\label{re-op1}
{\rm
Suppose that  the magnetic  field satisfies the condition \eqref{c-1} for any $\b(x_0, r) \subset \R^d$. 
Using the argument in  the proof of Theorem \ref{thm-op1} with $\Omega=\R^d$ or $\Omega=\R^d_+$, we obtain  
\begin{equation}\label{op-g1}
\aligned
c\int_{\R^d} \{ m(x, \B )  \}^2  |\psi|^2
 & \le \int_{\R^d} 
|(D+\A)\psi|^2,\\
\endaligned
\end{equation}
and
\begin{equation}\label{op-g1a}
c\int_{\R_+^d} \{ m(x, \B )  \}^2  |\psi|^2
  \le \int_{\R_+^d} 
|(D+\A)\psi|^2,\\
\end{equation}
for any $\psi \in C_0^1(\R^d; \C)$, where $c>0$ depends only on $d$ and $\eta$ in \eqref{c-1}.
}
\end{remark}

Let
\begin{equation}\label{bO}
\Omega_b =\left\{ x\in \Omega: \text{\rm dist} (x, \partial\Omega ) < \{ m(x,\B )\}^{-1}\right \}
\end{equation}
be a boundary layer of $\Omega$.

\begin{thm}\label{thm-op-2}
Let $\Omega$ be a bounded Lipschitz domain in $\R^d$.
Suppose $\B$ satisfies the condition \eqref{c-1} for any $\b(x_0, r) \subset \{ x\in \R^d: \text{\rm dist}(x, \partial\Omega)< r_0\}$.
Also assume that  $m(x, \B)> r_0^{-1}$ for any $x\in  \partial\Omega$. 
Then for any $\psi\in  C^1(\overline{\Omega}; \C)$,
\begin{equation}\label{op5-1}
 c\int_{\partial \Omega}
m(x, \B) |\psi|^2\, d\sigma 
\le  \int_{\Omega_b} 
|(D+\A)\psi|^2\, dx, 
\end{equation}
where  
 $c>0$ depends only on $d$, $\eta$ in \eqref{c-1} and the Lipschitz character of $\Omega$.
\end{thm}

\begin{proof}

Let $x_0 \in \partial\Omega$ and $r= c\{ m(x_0, \B) \}^{-1}< r_0$.
Then
$$
\aligned
\int_{\b (x_0, r) \cap \partial\Omega}
|\psi|^2\, d\sigma
 &\le  \frac{C}{r} \int_{\b (x_0,  2r ) \cap \Omega} |\psi|^2\, dx
 + C r \int_{\b (x_0, 2r) \cap \Omega}
\big |\nabla |\psi| \big|^2\, dx\\
 & \le  \frac{C}{r} \int_{\b x_0,  2r ) \cap \Omega} |\psi|^2\, dx
 + C r \int_{\b (x_0, 2r) \cap \Omega}
 | (D+\A) \psi |^2\, dx, 
\endaligned
$$
where we have used a trace inequality in a  Lipschitz domain for the first inequality and $\big|\nabla |\psi|\big |\le |(D+\A) \psi|$ for the second.
In view of \eqref{op4-2},  we obtain 
$$
\int_{\b (x_0, r) \cap \partial\Omega}
|\psi|^2\, d\sigma
\le C r \int_{\b (x_0, 6r) \cap \Omega}
 | (D+\A) \psi |^2\, dx.
$$
It follows that 
\begin{equation}\label{op5-2}
\aligned
 \int_{\b (x_0, r)\cap \partial\Omega } \{ m(x, \B)\}^d |\psi|^2 \, d\sigma
& \le 
 C \int_{\b (x_0, 6r)\cap \Omega} \{ m(x, \B)\}^{d-1} |(D+\A)\psi |^2\, dx,
\endaligned
\end{equation}
where we have used \eqref{m-2}. 
We now integrate both sides of \eqref{op5-2} with respect to $x_0$ on  $\partial\Omega$ and
then apply  Fubini's Theorem.
This, together with the observations that for $x\in \partial\Omega$, 
$$
|\{ x_0 \in \partial\Omega: |x_0 -x|< c \{ m(x, \B)\}^{-1} \} |\ge  c \{ m(x, \B )\}^{1-d},
$$
 and  that  for $x\in \Omega_b$,
 $$
 |\{ x_0 \in \partial\Omega: |x_0 -x|< 6 c \{ m(x, \B )\}^{-1} \} |  \le  C \{ m(x, \B )\}^{1-d},
 $$
 yields \eqref{op5-1}. 
\end{proof}

\begin{remark}\label{re-op2}
{\rm
Let
\begin{equation}
\Omega=\{ (x^\prime, x_d) \in \R^d:  x_d > \phi (x^\prime) \}
\end{equation}
be a graph domain in $\R^d$, where $\phi: \R^{d-1} \to \R$ is a Lipschitz function with $\|\nabla \phi\|_\infty\le M$.
Suppose that  the magnetic field  $\B$ satisfies the condition \eqref{c-1} for any $\b(x_0, r)\subset \R^d$.
Then
\begin{equation}\label{op-g1b}
c \int_{\partial\Omega} m(x, \B )  |\psi|^2
\le \int_\Omega |(D+\A) \psi|^2
\end{equation}
for any $\psi \in C_0^1(\R^d; \C)$, 
where $c>0$ depends only on $d$, $M$ and $\eta$ in \eqref{c-1}.
This follows from the proof of Theorem \ref{thm-op-2}.
}
\end{remark}


\section{Proofs of Theorems \ref{main-thm-1} and \ref{main-thm-2}}\label{sec-proof}

In this section we 
use the operator lower bounds obtained in Section \ref{section-op}
to prove lower bounds in \eqref{main-1a} and \eqref{main-2a} 
 for  the ground state energies. As a result, we establish the leading orders of
 $\lambda^D(\beta\A, \Omega)$, $\lambda^N(\beta\A, \Omega)$, and
$\lambda^{D\!N} (\beta\A, \Omega)$, as $\beta \to \infty$, under the conditions \eqref{main-c-1} and
\eqref{main-c-2}.

\begin{thm}\label{main-thm-1a}
Let $\Omega$ be a bounded Lipschitz domain in $\R^d$.
Suppose that the magnetic field $\B$ satisfies the conditions \eqref{c-1} for any $\b(x_0, r)\subset
\{ x\in \R^d: \text{\rm dist}(x, \Omega)< r_0\}$.
Also assume that there exist $\kappa, C_0, c_0>0$ such that 
 \begin{equation}\label{c-2}
\left(\fint_{\b(x, r)} |\B|^2 \right)^{1/2} \ge c_0 r^{\kappa}
\end{equation}
for any $0< r< r_0$ and any $x\in \overline{\Omega}$, and that 
\begin{equation}\label{c-3}
\left(\fint_{\b(x_0, r)} |\B|^2 \right)^{1/2} \le C_0 r^{\kappa}
\end{equation}
for any $0< r< r_0$ and some $x_0\in \overline{\Omega}$.
Then, for any $\beta>  C r_0^{-\kappa-2}$, 
\begin{equation}\label{main-1}
c\,  \beta^{\frac{2}{\kappa+2} }\le \lambda^N(\beta\A, \Omega)\le    \lambda^D (\beta \A, \Omega)  \le C  \beta^{\frac{2}{\kappa+2}},
\end{equation}
 where $C, c>0$ depend only on $d$, $\Omega$, $(c_0,C_0, \kappa)$ in \eqref{c-2}-\eqref{c-3} and  the function $\eta$ in \eqref{c-1}.
\end{thm}

\begin{proof}

The upper bound for $\lambda^D(\beta\A, \Omega)$  is  given by Theorem \ref{thm-u1}.
To prove the lower bound  for $\lambda^N (\beta\A, \Omega)$, we use Theorem \ref{thm-op1}.
Note that the magnetic field $\beta \B$ satisfies  the condition \eqref{c-1} with the  same $\eta$.
Also, under the condition \eqref{c-2}, we have
$$
\beta \max_{Q(x, r_0 )}  |\B| \ge c\beta r_0^\kappa > r_0^{-2}
$$
if  $x\in \overline{\Omega}$ and $\beta> C r_0^{-\kappa-2}$.
As a result, by Theorem \ref{thm-op1}, we obtain 
$$
c\int_\Omega
\{ m(x, \beta \B) \}^2 |\psi |^2\le 
\int_\Omega |(D+\beta \A)\psi|^2
$$
for any $\psi \in C^1(\overline{\Omega}; \C)$,
if $\beta>Cr_0^{-\kappa-2}$.
It follows that
\begin{equation}\label{pt-1}
\lambda^N(\beta\A, \Omega)\ge c \inf_{x\in {\Omega}} \{ m(x, \beta \B)\}^2.
\end{equation}
Since the condition \eqref{c-2} implies that
$$
  \max_{Q(x, r)} |\beta \B|
\ge c \beta  r^\kappa
$$
for any $x\in {\Omega}$ and $0< r< r_0$, 
by \eqref{m}, 
we have  $\{ m(x, \beta\B)\}^{-1} \le C \beta^{-\frac{1}{\kappa+2}}$.
Hence,  $$
m(x,  \beta \B)\ge c\,  \beta^{\frac{1}{\kappa+2}}
$$
 for any $x\in {\Omega}$.
 This, together with \eqref{pt-1}, gives 
 $$
 \lambda^D (\beta\A, \Omega) \ge \lambda^N (\beta\A, \Omega) \ge c\, \beta^{\frac{2}{\kappa+2}}
 $$
 for any $\beta>C r_0^{-\kappa-2}$.
\end{proof}

Let $ \mathcal{O}= \{ x\in \R^d: \text{\rm dist}(x, \partial\Omega)< r_0\}$.

\begin{thm}\label{main-thm-2a}
Let $\Omega$ be a bounded Lipschitz domain in $\R^d$.
Suppose $\B$ satisfies the conditions \eqref{c-1} for any $\b(x_0, r) \subset \mathcal{O}$. 
Also assume that  there exist $\kappa, C_0, c_0>0$ such that \eqref{c-2} holds for any $0<r<r_0$ and any $x\in \partial\Omega$, and 
that \eqref{c-3} holds for any $0<r<r_0$ and some $x_0\in \partial\Omega$.
Then, for any $\beta> C r_0^{-\kappa-2}$, 
\begin{equation}\label{main-2}
c\, \beta^{\frac{1}{\kappa+2} }\le \lambda^{D\!N}(\beta\A, \Omega) \le C \beta^{\frac{1}{\kappa+2}},
\end{equation}
 where $C>0, c>0$ depends only on $d$, $\Omega$,  $(c_0,C_0,  \kappa)$ in \eqref{c-2}-\eqref{c-3} and the function $\eta$ in \eqref{c-1}.
\end{thm}

\begin{proof}
The upper bound in \ref{main-2} is given by Theorem \ref{thm-u4}.
To prove the lower bound, we use Theorem \ref{thm-op-2}.
As in the proof of Theorem \ref{main-thm-1a}, this yields 
\begin{equation}\label{pt-2}
\lambda^{D\!N}(\beta \A, \Omega)
\ge c\inf_{x\in \partial\Omega} m (x, \beta \B).
\end{equation}
Finally, the condition \eqref{c-2} for any  $x\in \partial\Omega$ implies that 
$$
m(x, \beta \B)\ge c\,  \beta^{\frac{1}{\kappa+2}}
$$
 for any $x\in \partial\Omega$.
\end{proof}

\begin{proof}[Proof of Theorem \ref{main-thm-1}]

By extension, with the loss of generality,  we may assume that $\A\in C^\infty(\R^d; \R^d)$.
By continuity, we may also assume the condition \eqref{main-c-1} holds for any $x\in \widetilde{\Omega}
=\{x\in \R^d: \text{dist}(x, \Omega)< r_0\}$.
As a result, the condition  \eqref{c-1} holds for any $\b(x_0, r) \subset  \widetilde{\Omega}$.
See Remark \ref{re-op-2}. 
Moreover, by  Taylor's Theorem,  the condition \eqref{c-2} with $\kappa=\kappa_* $ holds for any $x\in \overline{\Omega}$, and
the condition \eqref{c-3} with $\kappa=\kappa_*$ holds for any $x$ with $\kappa(x)=\kappa_*$.
Consequently, Theorem \ref{main-thm-1} follows readily from Theorem \ref{main-thm-1a}.
\end{proof}

\begin{proof}[Proof of Theorem \ref{main-thm-2}]

By continuity,  we may assume \eqref{main-c-2} holds for any $x\in \mathcal{O}$.
It follows that the condition \eqref{c-1} holds for any $\b(x_0, r)\subset \mathcal{O}$.
As a result, Theorem \ref{main-thm-1} follows from Theorem \ref{main-thm-2a}.
\end{proof}


\section{Polynomial magnetic fields}\label{section-p}

Throughout this section we assume that the magnetic potential $\A$ is a (vector-valued) homogeneous polynomial of degree
$\kappa+1$. Thus, 
  the magnetic field
$\B=\nabla \times \A$ is a (matrix-valued) homogeneous polynomial of degree $\kappa$; i.e., 
\begin{equation}\label{hb}
\B(x)=\sum_{|\alpha|=\kappa} b_\alpha x^\alpha
\end{equation}
for some constant $d\times d$ matrices $\{b_\alpha\}$.
We also  assume that $\sum_{|\alpha|=\kappa} |b_\alpha|=1$.
Note  that  such $\B$ satisfies the condition \eqref{c-1} for any $\b(x_0, r)\subset \R^d$
with  $\eta(t)=C_0t$, where  $C_0$ depends only on $d$ and $\kappa$.
Moreover,
\begin{equation}\label{mp}
c\,  \widetilde{m}(x, \B)\le m(x, \B) 
\le C  \widetilde{m}(x, \B),
\end{equation}
for some $C, c>0$ depending only on $d$ and $\kappa$, where
\begin{equation}\label{mp-1}
\widetilde{m}(x, \B)=
 \sum_{|\alpha|\le \kappa} |\partial^\alpha \B(x)|^{\frac{1}{|\alpha|+2}}.
\end{equation}
The inequalities in \eqref{mp} follow  from the observation that if $\P$ is a polynomial of degree $\kappa$, then
$$
 c\max_{Q(x_0, r)} |\P|
\le
\sum_{|\alpha|\le \kappa} |\partial^\alpha \P(x_0)| r^{|\alpha|}
\le C \max_{Q(x_0, r)} |\P|,
$$
for any $x_0\in \R^d$ and $r>0$, 
where $C, c>0$ depend only on $d$ and $\kappa$.
In particular, by \eqref{mp} and the assumption $\sum_{|\alpha|=k}  |b_\alpha|=1$, 
\begin{equation}\label{m-l}
m(x, \B) \ge c \sum_{|\alpha|=\kappa} | b_\alpha|^{\frac{1}{\kappa+2}}\ge c_0, 
\end{equation}
where $c_0>0$ depends only on $d$ and $\kappa$.
This, together with \eqref{op-g1}, \eqref{op-g1a} and \eqref{op-g1b}, shows that 
\begin{equation}\label{low-a}
c\le \lambda (\A, \R^d), \lambda^D (\A, \R^d_+), \lambda^N(\A, \R^d_+), \lambda^{D\!N} (\A, \R^d_+) \le C,
\end{equation}
where $C, c>0$ depend only on $d$ and $\kappa$.

\begin{defn}
Let $\B$ be given by \eqref{hb}.
The  set 
\begin{equation}\label{md}
\left\{ y \in \R^d: \ \B (x+y)=\B (x) \quad \text{ for any } x\in \R^d \right\}
\end{equation}
is called the invariant subspace for $\B$.
\end{defn}

\begin{prop}\label{prop-1}
Let $\B$ be a homogeneous polynomial given by \eqref{hb}. 
The invariant subspace for $\B$  is given by  
\begin{equation}
V=\left\{ y \in \R^d: \langle y, \nabla \partial^\alpha B_{j\ell } (0) \rangle =0  \text{ for any } 1\le j < \ell  \le d 
\text{ and } |\alpha|=\kappa-1 \right\}
\end{equation}
\end{prop}

\begin{proof}
Let $y\in \R^d$.
Suppose $\B(x+y)=\B(x)$ for any $x\in \R^d$.
Since
$$
\B(x + t y )= t^\kappa \B(t^{-1} x +y) = t^\kappa \B  (t^{-1} x) =\B (x)
$$
for any  $t\in \R$ and $t \neq 0$, it follows that $\partial_t  \{ \partial_x^\alpha B_{j\ell } (x+ty) \}=0$
for any $1\le j< \ell\le d$ and $|\alpha|=\kappa-1$.
Hence, $ \langle y, \nabla \partial^\alpha B_{j\ell } (0) \rangle =0$   for any  $ 1\le j < \ell  \le d$ and  
 and  $|\alpha|=\kappa-1$.

Suppose $y\in V$ and $y\neq 0$.
 Let $e_1, e_2, \dots, e_d$ be an orthonormal basis for $\R^d$ with $e_1= y/|y|$.
 Write
 \begin{equation}\label{B-t}
 \B (x)= \sum_{|\alpha|=\kappa} \widetilde{b}_\alpha \langle x, e_1 \rangle^{\alpha_1} \cdots
 \langle x, e_d \rangle^{\alpha_d},
 \end{equation}
 where $\alpha=(\alpha_1, \dots, \alpha_d)$ and $\widetilde{b}_\alpha = \partial_{e_1}^{\alpha_1} \cdots \partial_{e_d}^{\alpha_d} \B (0)/\alpha!$.
 Since $\partial_{e_1} \partial ^\alpha B (0)=0$  for any $|\alpha|=\kappa-1$, 
 it follows that $\widetilde{b}_\alpha=0$ if $\alpha_1\neq 0$.
 As a result, by \eqref{B-t}, we obtain $\B(x+y)= \B (x)$ for any $x\in \R^d$.
\end{proof}

Clearly,  if  the  invariant subspace $V=\R^d$, then $\B$ is constant and $\kappa=0$.
Suppose that  dim$(V)<d$. 
Then 
\begin{equation}\label{mp-1b}
\min_{\substack{z\in V^\perp\\ |z|=1} }\sum_{|\alpha|= \kappa-1}\sum_{j, \ell} 
|\langle z, \nabla \partial^\alpha B_{j\ell} (0)\rangle | = \sigma>0.
\end{equation}
Note that if $|\alpha|=\kappa-1$,  we have $\partial^\alpha B_{j\ell} (z)= \langle z, \nabla \partial^\alpha B_{j \ell}(0)\rangle $.
It follows that if $x=y+z$, where $y\in V$ and $z\in V^\perp$, then
$$
\aligned
\widetilde{m}(x, \B) & =\widetilde{m}(z, \B )\ge
\sum_{|\alpha|= \kappa-1} |\partial ^\alpha \B(z)|^{\frac{1}{\kappa+1}}\\
&=\sum_{|\alpha|= \kappa-1}
|z|^{\frac{1}{\kappa+1}}
 |\partial^\alpha \B(z/|z|)|^{\frac{1}{\kappa+1}}\\
&\ge c\, \sigma^{\frac{1}{\kappa+1} }  |z|^{\frac{1}{\kappa+1}},
\endaligned
$$
provided $|z|> 0$.
This shows that
\begin{equation}\label{mp-1c}
\widetilde{m}(x, \B) \ge c\,  \sigma^{\frac{1}{\kappa+1}}  \left\{ \text{dist}(x, V)\right\}^{\frac{1}{\kappa+1}}
\end{equation}
for any $x\in \R^d$, 
where $\sigma>0 $ is given by \eqref{mp-1b}.

\begin{lemma}\label{lemma-p1}
Let $\A$ be a homogeneous polynomial of degree $\kappa+1$.
Then,  for $R> 1$, 
\begin{equation}\label{p1-0}
\lambda(\A, \R^d)  \le \lambda^D (\A, \b (0, R)) \le \lambda (\A, \R^d)  + C R^{-2},
\end{equation}
where $C$ depends only on  $d$, $\sigma$ in \eqref{mp-1b} and  $\kappa$.
\end{lemma}

\begin{proof}

The first inequality in \eqref{p1-0}  is obvious,  since $C_0^1(\b(0, R); \C)\subset C_0^1(\R^d; \C)$.
To show the second inequality,  we first note that 
$$
\aligned
\lambda^D(\A, \b(0, R))  & \le \lambda^D (\A, \b(0, 1)) \\
& \le C (1+ \| \B \|_{L^\infty(\b(0, 1))}) \le C.
\endaligned
$$
Thus, we only need to consider the case  where $R>1$ is large.
Let $\{ \varphi_\ell \}$ be a sequence of functions  such that
$
\sum_{\ell=1}^\infty \varphi_\ell^2 =1  \text{ in } \R^d,
$
where  $\varphi_\ell \in C_0^\infty(\b (x_\ell, R); \R)$, $|\nabla \varphi_\ell |\le CR^{-1}$, and $\sum_\ell \chi_{\b (x_\ell, R)}
\le C$.
Using the identity, 
\begin{equation}\label{partition}
\int_{\R^d} |(D+\A)\psi|^2
=\sum_\ell \int_{\R^d} |(D+\A)(\psi \varphi_\ell)|^2
-\sum_\ell \int_{\R^d} |\nabla \varphi_\ell|^2 |\psi|^2,
\end{equation}
we obtain
$$
\aligned
\int_{\R^d} |(D+\A)\psi|^2
 & \ge \sum_\ell  \lambda^D (\A, \b (x_\ell, R)) \int_{\R^d} |\varphi_\ell  \psi|^2
- C \sum_\ell R^{-2} \int_{\b (x_\ell, R)} |\psi|^2\\
&\ge  \left\{  \inf_{x\in \R^d}  \lambda^D (\A, \b (x, R) )- CR^{-2} \right\}
\int_{\R^d} |\psi|^2
\endaligned
$$
for any $\psi \in C_0^1(\R^d; \C)$.
It follows that  
\begin{equation}\label{p1-3}
\lambda(\A, \R^d)  \ge \inf_{x\in \R^d} \lambda^D (\A, \b (x, R)) -CR^{-2},
\end{equation}
where $C$ depends only on $d$.

Next,
let $V$ be the invariant  subspace for $\B$.
Note  that 
$$
\lambda^D  (\A, \b (x+y, R))=\lambda^D (\A, \b (x, R))
\quad \text{ for any } x\in \R^d \text{ and } y\in V.
$$
As a result, it suffices to  show that 
$$
\lambda^D (\A, \b (z, R)) \ge \lambda^D (\A, \b (0, 3R)).
$$
 for any $z\in V^\perp$.
To this end, we observe that by \eqref{mp} and \eqref{op-g1}, 
$$
c \int_{\R^d}  \{ \widetilde{m} (x, \B)\}^2 |\psi|^2
\le \int_{\R^d} |(D+ \A)\psi|^2.
$$
In view of \eqref{mp-1c}, for $z\in V^\perp$, we obtain 
$$
\aligned
\lambda^D (\A, \b (z, R))& \ge c\inf_{x\in \b (z, R)} \{ \widetilde{m} (x, \B)\}^2 \\
& \ge c\, \sigma^{\frac{2}{\kappa+1}}  \inf_{x\in \b (z,R)} \left\{ \text{\rm dist}(x, V)\right\}^{\frac{2}{\kappa+1}}\\
&\ge  c\, \sigma^{\frac{2}{\kappa+1}}   \left\{ \text{\rm dist}(z, V) - R \right\}^{\frac{2}{\kappa+1}}\\
&\ge \lambda^D(\A, \b (0, 1))\ge \lambda^D (\A, \b (0, 3R)),
\endaligned
$$
if $\text{\rm dist}(z, V) \ge 2R $ and $R>C $ is large.
Here we also use the fact $\lambda^D(\A, \b(0,1))\approx 1$,
under the normalization $\sum_{|\alpha|=\kappa} |b_\alpha|=1$.

Finally, if $z\in V^\perp$ and dist$(z, V) =|z|  < 2R$, then   $\b(z, R) \subset \b (0, 3R)$. Hence,
$$
\lambda^D (\A, \b (z, R) ) \ge \lambda^D (\A, \b (0, 3R)), 
$$
which completes the proof.
\end{proof}

For $x\in \R^{d-1} \times \{ 0 \}=\partial \R^d_+$, let
\begin{equation}\label{p-i}
Q_+(x, R)= Q(x, R)\cap \R^d_+.
\end{equation}

\begin{lemma}\label{lemma-p2}
Let  $\A$ be  a homogeneous polynomial of degree $\kappa+1$.
Then,  
\begin{equation}\label{p2-0}
 \lambda^D (\A, \R^d_+) \le \lambda^D (\A, Q_+(0, R))
  \le  \lambda^D (\A, \R^d_+) +  CR^{-2} 
\end{equation}
for $R> 1$, where $C$ depends on $\B$.
\end{lemma}

\begin{proof}
Since $Q_+(0, R) \subset \R^d_+$, we have $\lambda^D (\A, \R^d_+) \le \lambda^D (\A, Q_+ (0, R))\le C $ for any $R>1$.
To show 
\begin{equation}\label{p2-1a}
\lambda^D  (\A,  Q_+(0, R))\le  \lambda^D (\A, \R^d_+) +  CR^{-2}
\end{equation}
for $R$ large, 
let $V$ be the invariant subspace for $\B$ and $k=\text{\rm dim}(V)$.
If $k=d$, then $\B$ is constant.
The estimate follows readily by using a partition of unity.
Assume dim$(V) \le d-1$.
We construct a parallelotope, 
\begin{equation}\label{p2-p}
P=\big\{  t_1 e_1 + t_2 e_2+ \cdots +  t_d e_d: 0< t_j <1 \text{ for } 1\le j \le d \big \} \subset \R^d_+,
\end{equation}
where $\{e_1, e_2, \dots, e_d \}$ are unit vectors that  form a basis for $\R^d$.
These  unit vectors are constructed  as follows.

\begin{itemize}

\item

If $V \subset \R^{d-1}\times \{ 0\}$,  choose $\{ e_1, e_2, \dots, e_{d-1}\}$  to  form  an orthonormal basis for $\R^{d-1}\times \{0\}$ 
so that   $\{ e_1, \dots, e_k\}$ is an orthonormal basis for $V$.
Also, choose  $e_d =(0, \dots,0,  1) \in \R^d_+ \cap  V^\perp$.
In this case, $P$ is a cube.

\item

If  $V \not\subset \R^{d-1}\times \{ 0\} $, choose $\{e_1, e_2, \dots, e_{d-1} \}$ to form an orthonormal basis for $\R^{d-1} \times \{0\}$
so that $\{ e_1, \dots, e_{k-1} \}\subset V$.
Also, choose  $e_d\in V \cap \R^{d}_+$ such that  $\{e_1, \dots, e_{k-1}, e_d\}$ forms an orthonormal basis for $V$.

\end{itemize}
As a result, we obtain a tiling of $\R^d_+$, 
\begin{equation}\label{p2-b}
\overline{\R^d_+ }  = \bigcup_{z\in \mathcal{Z}} \overline{ R  (z+ P)},
\end{equation}
where
$$
\mathcal{Z}= \big\{ n_1 e_1 + n_2 e_2 +\cdots  + n_d e_d: n_j \in \mathbb{Z} \text{ for } 1\le j \le d \text{ and } n_d \ge 0 \big\}
$$
forms  a lattice for $\overline{\R^d_+}$.
Let 
\begin{equation}\label{p2-p1}
\widetilde{P}= \big\{ t_1 e_1 + t_2 e_2 + \cdots + t_d e_d: 0< t_j < 2  \text{ for } 1\le j \le d  \big\}.
\end{equation}
By using a partition of unity adapted to \eqref{p2-b} and \eqref{p2-p1},  we obtain 
\begin{equation}\label{p2-1}
\aligned
\lambda^D (\A, \R^d_+)
 & \ge \inf_{z\in \mathcal{Z}} \lambda^D (\A, R (z+ \widetilde{P} )) - C R^{-2}\\
   & = \inf_{z\in \mathcal{Z} _1} \lambda^D (\A,  R (z+ \widetilde{P})) - C R^{-2},
  \endaligned
\end{equation}
where  $\mathcal{Z}_1\subset \mathcal{Z}$ is given in the following, by 
using the fact that $\B$ is  invariant with  respect to $V$ to eliminate the components of $z$  in $V$.
If $V\subset \R^{d-1}\times \{ 0\}$, the set $\mathcal{Z}_1$ is given by
$
\mathcal{Z}_1 = \mathcal{Z}\cap V^\perp.
$
If $V\not\subset \R^{d-1}\times \{ 0\}$,
$$
\mathcal{Z}_1
=\big\{ n_k e_k + \cdots + n_{d-1} e_{d-1} \in \mathcal{Z}: n_j \in \mathbb{Z}  \text{ for } k\le j \le d-1  \big\}.
$$

Note that  for $z\in \mathcal{Z}_1$, 
$$
\aligned
\lambda^D (\A,  R (z+\widetilde{P}))
 & \ge c \inf_{x\in  R (z+\widetilde{P})}   
\{ \widetilde{m}(x, \B) \}^2\\
& \ge c\, \sigma^{\frac{2}{\kappa+1}}  \inf_{ x\in  R (z+\widetilde{P})}  \{ \text{\rm dist}(x, V)\}^{\frac {2}{\kappa+1}}\\
 & \ge c\,  \sigma^{\frac{2}{\kappa+1}} R^{\frac{2}{\kappa+1}} \{ \text{\rm dist}(z, V) - 2d  \}^{\frac {2}{\kappa+1}}.
\endaligned
$$
It  follows that if $z\in \mathcal{Z}_1$ and dist$(z, V) \ge  2d+1$, we have
$$
\lambda^D (\A,  R  (z+\widetilde{P}))
\ge \lambda^D (\A, Q_+ (0, 1))
$$
if $R>1$ is large.
If  $z\in \mathcal{Z}_1 $ and dist$(z, V)< 2d+1 $, we claim that 
\begin{equation}\label{claim-1}
 R (z+\widetilde{P}) \subset Q_+ (0, C R).
 \end{equation}
 This implies  that  for any  $z\in \mathcal{Z}_1$,
$$
\lambda^D (\A,  R (z+ \widetilde{P}))
\ge \lambda^D (\A, Q_+(0,  CR)).
$$
which, together with \eqref{p2-1}, gives \eqref{p2-1a}.

Finally, to prove the claim \eqref{claim-1}, we consider two cases.
If $V\subset \R^{d-1} \times \{ 0\}$, then $\mathcal{Z}_1= Z\cap V^\perp$. Hence,
$|z| =\text{\rm dist}(z, V)< 2d+1$, which  yields \eqref{claim-1}.
If $V\not\subset  \R^{d-1}\times\{ 0\}$ and $z\in \mathcal{Z}_1$, then 
$
z=n_k e_k +\cdots+ n_{d-1} e_{d-1}\in \R^{d-1}\times \{ 0\}.
$
Note that  if $|z|\neq 0$, 
$$
\aligned
\{ \text{\rm dist} (z, V) \}^2 & = |z|^2 - \langle z, e_1\rangle^2 -\cdots - \langle z, e_{k-1}\rangle^2 - \langle z, e_d \rangle^2\\
& =|z|^2- \langle z, e_d \rangle^2
=|z|^2 \left\{ 1-  \langle z/|z|, e_d \rangle^2\right\} \\
&\ge |z|^2 \{ \text{\rm dist}(e_d, \R^{d-1} \times \{ 0 \} ) \}^2
\ge c|z|^2,
\endaligned
$$
where we have used the facts that $\{ e_1, \dots, e_{k-1}, e_d\}$  forms an orthonormal basis for $V$ and 
$e_d\not \in \R^{d-1} \times \{ 0 \}$.
It follows that  $|z|\le C$, which leads to  \eqref{claim-1}.
\end{proof}

\begin{remark}\label{re-de-1}
{\rm 
The constants $C$ in \eqref{p2-1} and \eqref{claim-1} depend on $\sigma$ and  the shape of $P$.
Indeed, in the case $V\not\subset  \partial\R^d_+$,
 the constant $C$ in \eqref{p2-0} depends on $\sigma$  and
 $\max_{v\in V} |\langle v, n \rangle|$, where $n=(0, \dots, 0, -1)$ is the outward normal to $\partial\R^d_+$.
 This follows from the fact $|\langle v, n\rangle | =\text{\rm dist}(v, \partial \R^d_+)$.
}
\end{remark}

Let $\mathcal{O}$ and $\Omega$ be two Lipschitz domains in $\R^d$ such that  $\mathcal{O}\subset \Omega$ and
$\partial\mathcal{O}\cap \partial\Omega \neq \emptyset$.
Define

\begin{equation}\label{p-i-1}
\aligned
\mu^N (\A, \mathcal{O}, \Omega)  & = \inf _{\psi} 
\frac{\int_{\mathcal{O}} |(D+\A)\psi|^2}{\int_{\mathcal{O} } |\psi|^2}, \\
\mu^{D\!N} (\A, \mathcal{O}, \Omega )  & = \inf_{\psi }
\frac{\int_{\mathcal{O}} |(D+\A)\psi|^2}{\int_{\partial \mathcal{O} \cap \partial \Omega} |\psi|^2},
\endaligned
\end{equation}
where the  infimums   are  taken over  non-zero  functions in $C^1(\overline{\mathcal{O}};  \C)$ such that
$\psi =0 $ on $\Omega \cap \partial \mathcal{O}$.
Observe that for $\Omega$ fixed, both $\mu^N (\A, \mathcal{O}, \Omega)$ and
$\mu^{D\!N}(\A, \mathcal{O}, \Omega)$  are monotonic with respect to  $\mathcal{O}$.
Indeed, if $\mathcal{O}_1\subset \mathcal{O}_2\subset\Omega$, then
$\mu^N (\A, \mathcal{O}_2, \Omega) \le \mu^N (\A, \mathcal{O}_1, \Omega)$ and
$\mu^{D\!N} (\A, \mathcal{O}_2, \Omega) \le \mu^N (\A, \mathcal{O}_1, \Omega)$.

\begin{lemma}\label{lemma-p3}
Let  $\A$ be  a homogeneous polynomial of degree $\kappa+1$.
Then,  
\begin{equation}\label{p3-0}
 \lambda^N (\A, \R^d_+) \le  \mu^N (\A, Q_+(0, R), \R^d_+)
  \le  \lambda^N (\A, \R^d_+)  +  CR^{-2} ,
\end{equation}
for $R> 1$, where $C$ depends on $\B$.
\end{lemma}

\begin{proof}
The first  inequality in \eqref{p3-0} follows readily from the definition.
The proof for the second is  similar to  that of  Lemma \ref{lemma-p2}.
The translation argument  for interior parallelotopes uses  the observation that 
$\mu^N (\A, \mathcal{O}, \R^d_+) \le \lambda^D (\A, \mathcal{O})$.
\end{proof}

\begin{lemma}\label{lemma-p4}
Let  $\A$ be  a homogeneous polynomial of degree $\kappa+1$.
Then,  
\begin{equation}\label{p4-0}
 \lambda^{D\!N} (\A, \R^d_+) \le  \mu^{D\!N} (\A, Q_+(0, R), \R^d_+)
  \le  \lambda^{D\!N }(\A, \R^d_+)  +  CR^{-2} ,
\end{equation}
for $R>1$, where $C$ depends on $\B$.
\end{lemma}

\begin{proof}
The first inequality in \eqref{p4-0} follows readily  from the definitions.
The proof of the second  uses the  inequality, 
\begin{equation}\label{p4-1}
 c\int_{\R^{d-1}}
\widetilde{m} (x, \B) |\psi |^2
\le  \int_{\R^d_+} |(D+\A)\psi|^2.
\end{equation}
(see Remark \ref{re-op2})
as well as the basis $\{e_1, e_2, \dots, e_d\}$  for $\R^d$, constructed in the proof of Lemma \ref{lemma-p2}.
Let $\psi \in C_0^1(\R^d; \C)$,
by using a partition of unity adapted  to \eqref{p2-b} and \eqref{p2-p1}, 
$$
\sum_{z\in \mathcal{Z}}  \int_{\R^d_+} |(D+\A)(\psi \varphi_z )|^2
\le \int_{\R^d_+} |(D+\A)\psi|^2
+  C R^{-2} \int_{\R^d_+} |\psi |^2,
$$
we obtain
\begin{equation}\label{p4-2}
\aligned
 \inf_{z\in \mathcal{Z}_0}  & \mu^{D\!N} (\A, R(z+\widetilde{P}), \R^d_+) \int_{\R^{d-1}} |\psi|^2\\
& \le \sum_{z\in \mathcal{Z}_0}
\mu^{D\!N} (\A, R(z+\widetilde{P}), \R^d_+) \int_{\R^{d-1}} |\psi  \varphi_z|^2\\
&\le \sum_{z\in \mathcal{Z}_0}
\int_{\R^d_+} |(D+\A) (\psi \varphi_z) |^2\\
&\le \int_{\R^d_+} |(D+\A)\psi|^2
+  C R^{-2} \int_{\R^d_+} |\psi |^2\\
& \le \left( 1+ C R^{-2}  [ \lambda^N (\A, \R^d_+) ]^{-1} \right)
\int_{\R^d_+} | (D+\A)\psi)|^2,
\endaligned
\end{equation}
where
$
\mathcal{Z}_0 
= \mathcal{Z} \cap  (\R^{d-1} \times \{ 0\} ).
$ 
We will show  that  for any $z\in \mathcal{Z}_0$,
\begin{equation}\label{p4-3}
\mu^{D\!N} (\A, R(z+ \widetilde{P}), \R^d_+)
\ge \mu^{D\!N} (\A, Q_+(0,  CR), \R^d_+).
\end{equation}
This, together with \eqref{p4-2}, implies that
$$
\lambda^{D\!N} (\A, \R^d_+)
\ge
\frac{\mu^{D\!N} (\A, Q_+ (0, C R), \R^d_+)}
{ 1+ C R^{-2} [\lambda^N (\A, \R^d_+)]^{-1} },
$$
which  gives the second inequality in \eqref{p4-0}, using the observation that $\lambda^N(\A, \R^d_+)\approx  1$.

To prove  \eqref{p4-3}, we use the same argument as  in the proofs of the last two lemmas.
Since $\B$ is invariant with respect to $V$, we may assume $z\in \mathcal{Z}_1$, where
$\mathcal{Z}_1\subset \mathcal{Z}$ is  the same as in the proof of Lemma \ref{lemma-p2}. 
For $z\in \mathcal{Z}_1 $, we use \eqref{p4-1} to obtain 
$$
\aligned
\mu^{D\!N}(\A, R(z+\widetilde{P}), \R^d_+)
 & \ge c \inf_{x\in R(z+ \widetilde{P})}
\widetilde{m}(x, \B)\\
&\ge c\, \sigma^{\frac{1}{\kappa+1}}  \inf_{x\in R(z+\widetilde{P})} \{ \text{\rm dist}(x, V) \}^{\frac{1}{\kappa +1}}\\
&\ge c\, \sigma^{\frac{1}{\kappa+1}}  R^{\frac{1}{\kappa+1}} \{ \text{\rm dist} (z, V) - 2d \}^{\frac{1}{\kappa+1}}.
\endaligned
$$
Hence, if  $z\in \mathcal{Z}_1$ and  dist$(z, V)\ge 2d+1$,  we have
$$
\mu^{D\!N}(\A, R(z+\widetilde{P}), \R^d_+)
\ge \mu^{D\!N}(\A, Q_+(0, 1), \R^d_+)
$$
for $R>1$ large.
If $z\in \mathcal{Z}_1$ and dist$(z, V)< 2d+1$, then $R(z+\widetilde{P})\subset Q_+ (0, C R)$ and  thus
$$
\mu^{D\!N} (\A, R(z+\widetilde{P}), \R^d_+)
\ge \mu^{D\!N} (\A, Q_+ (0, C R), \R^d_+).
$$
This completes the proof.
\end{proof}

Let 
\begin{equation}\label{omega}
\Omega=\big \{ (x^\prime, x_d)\in \R^d: \ x_d >\phi(x^\prime) \big\}, 
\end{equation}
where $\phi: \R^{d-1} \to \R$ is a $C^{1}$ function such that $\phi(0)=0$ and $\nabla \phi (0)=0$.

\begin{thm}\label{thm-p5}
Let $\A$ be a homogeneous polynomial of degree $\kappa+1$ and $\Omega$  given by \eqref{omega}.
Then
\begin{equation}\label{p5-0}
|\lambda^D (\A, \Omega\cap Q(0, 2R))
-\lambda^D ( \A ,  \R^d_+) |\le C\left \{ R^{\kappa+1} M_R + R^{-2} \right\},
\end{equation}
where $R>1$ and $M_R=\max \{ |\nabla \phi (x^\prime)|: |x^\prime|< R \}$.
\end{thm}

\begin{proof}

We assume $M_R\le (1/2)$ and $R$ is large,   for otherwise the estimate is trivial.
Since $\phi (0)=0$,  we have $|\phi (x^\prime) |\le R M_R$ for $|x^\prime|< R$.
It follows that $F(R) \subset \Omega\cap Q(0, 2R)\subset E(R)$, where
$$
\aligned
E(R) & =\big\{ (x^\prime, x_d): \ |x^\prime|< R \text{ and } -R M_R < x_d< R\big \}, \\
F(R) & =\big\{ (x^\prime, x_d): \ |x^\prime|< R \text{ and } R M_R < x_d< R \big\}.
\endaligned
$$
Hence,
$$
\lambda^D (\A, E(R))\le \lambda^D (\A, \Omega\cap Q(0, 2R))
\le \mu^D (\A, F(R)).
$$
By translation,
$$
\lambda^D (\A, F(R)) =\lambda^D (\widetilde{\A}, \widetilde{F}(R)),
$$
where $\widetilde{A} (x)=\A (x+ (0, \dots, 0, RM_R))$ and
$\widetilde{F}(R) = F(R) - (0, \dots, 0, RM_R)$.
Note that for $x\in \widetilde{F}(R)$,
$$
|\widetilde{\A}(x) -\A(x)|
\le C R^{\kappa+1} M_R.
$$
It follows that
$$
|\lambda^D (\widetilde{\A}, \widetilde{F}(R))
-\lambda^D (\A, \widetilde{F}(R))|
\le C R^{\kappa+1} M_R.
$$
As a result,
$$
\aligned
\lambda^D (\A, \Omega\cap Q(0, 2R))
& \le \lambda^D(\A, \widetilde{F}(R)) + C R^{\kappa+1} M_R\\
& \le \lambda^D (\A, Q_+ (0, R)) +C R^{\kappa+1} M_R\\
&\le  \lambda^D(\A, \R^d_+) + CR^{-2} + C R^{\kappa+1} M_R,
\endaligned
$$
where we have used the fact $Q_+(0, R) \subset \widetilde{F}(R)$ and Lemma \ref{lemma-p2}.
The lower bound for $\lambda^D(\A, \Omega\cap Q(0, R))$ may be established in 
a similar manner, using $E(R)$. We omit the details. 
\end{proof}

\begin{thm}\label{thm-p6}
Let $\A$ be a homogeneous polynomial of degree $\kappa+1$ and $\Omega$ given by \eqref{omega}.
Then
\begin{equation}\label{p6-0}
|\mu^N (\A, \Omega\cap Q(0, 2R), \Omega)
-\lambda^N ( \A ,  \R^d_+) |\le C \{ R^{\kappa+1} M_R + R^{-2} \},
\end{equation}
where $R>1$ and $M_R=\max \{ |\nabla \phi (x^\prime)|: |x^\prime|< R \}$.
\end{thm}

\begin{proof}
Assume $M_R< (1/2)$.
Consider the map
$$
\Phi: (x^\prime, x_d ) \to (x^\prime, x_d -\phi (x^\prime)),
$$
which flattens the boundary of $\Omega$.
Let $\psi\in C_0^1( Q(0, R); \C)$ and
$\varphi (x)=\psi (x^\prime, x_d + \phi(x^\prime))$.
Note that
$$
\aligned
&\left(\int_{\Omega\cap Q (0, 2R)} |(D+\A) \psi|^2\right)^{1/2}\\
&\le \left(\int_{\Omega\cap Q(0, 2R)}
|(D+\widetilde{\A}) \varphi (\Phi(x))|^2 \right)^{1/2}
+ M_R \left(\int_{\Omega\cap Q(0, 2R)}
|D\varphi (\Phi (x))|^2 \right)^{1/2}\\
& \le (1+M_R) \left(\int_{\Omega\cap Q(0, 2R)}
|(D+\widetilde{\A}) \varphi (\Phi(x))|^2 \right)^{1/2}
+ M_R \left(\int_{\Omega\cap Q(0, 2R)}
|\widetilde{\A} \varphi (\Phi (x))|^2 \right)^{1/2}\\
&  \le (1+M_R) \left(\int_{\R^d_+\cap \Phi(Q(0, 2R))}
|(D+\widetilde{\A}) \varphi )|^2 \right)^{1/2}
+ M_R \left(\int_{\R^d_+\cap \Phi (Q(0, 2R))}
|\widetilde{\A} \varphi |^2 \right)^{1/2}\\
 & \le (1+M_R) \left(\int_{\R^d_+\cap \Phi(Q(0, 2R))}
|(D+{\A}) \varphi )|^2 \right)^{1/2}
+ M_R \left(\int_{\R^d_+\cap \Phi (Q(0, 2R))}
|\widetilde{\A} \varphi |^2 \right)^{1/2}\\
& \qquad\qquad
+ (1+M_R) \left(\int_{\R^d_+\cap \Phi (Q(0, 2R))}
|(\A-\widetilde{\A} ) \varphi |^2 \right)^{1/2}, \\
\endaligned
$$
where $\widetilde{\A} (x) =\A (\Phi^{-1}(x))$.
It follows that
$$
\aligned
\mu^N (\A, \Omega\cap Q(0, 2R), \Omega)
& \le  (1+ M_R)
\mu^N (\A, \R^d_+\cap \Phi(Q(0, 2R)), \R^d_+)
+ C M_R R^{\kappa+1} \\
&\le \mu^N(\A,  Q_+(0, cR), \R^d_+)
+ CM_R R^{\kappa+1}\\
&\le \lambda^N(\A, \R^d_+) + C R^{-2} + C M_R R^{\kappa+1},
\endaligned
$$
where we have used the fact that  $Q(0, cR)\subset \Phi(Q(0, 2R))$ for some small $c>0$ for the second inequality 
and Lemma \ref{lemma-p3} for the third.
The lower bound for $\mu^N (\A, \Omega\cap Q(0, 2R), \Omega)$ may be proved by a similar perturbation argument.
\end{proof}

\begin{thm}\label{thm-p7}
Let $\A$ be a homogeneous polynomial of degree $\kappa+1$. and $\Omega$  given by \eqref{omega}.
Then
\begin{equation}\label{p7-0}
|\mu^{D\!N} (\A, \Omega\cap Q(0, 2R), \Omega)
-\lambda^{D\!N} ( \A ,  \R^d_+) |\le C \{ R^{\kappa+1} M_R + R^{-2} \},
\end{equation}
where $R>1$ and $M_R=\max \{ |\nabla \phi (x^\prime)|: |x^\prime|< R \}$.
\end{thm}

\begin{proof}

Assume $M_R<(1/2)$.
Let $\psi$, $\varphi$ and $\Phi$ be the same as in the proof of Theorem  \ref{thm-p6}.
It follows that
$$
\aligned
&\left(\int_{\Omega\cap Q (0, 2R)} |(D+\A) \psi|^2\right)^{1/2}\\
& \le 
(1+M_R) \left(\int_{\R^d_+\cap \Phi(Q(0, 2R))}
|(D+ {\A}) \varphi )|^2 \right)^{1/2}
+C  M_R R^{\kappa+1}  \left(\int_{\R^d_+\cap \Phi (Q(0, 2R))}
| \varphi |^2 \right)^{1/2}\\
&\le (1+ C M_R R^{\kappa+1})
\left(\int_{\R^d_+\cap \Phi (Q(0, 2R))}
| (D+\A) \varphi |^2 \right)^{1/2}.
\endaligned
$$
Also, note that
$$
\aligned
 & \int_{\partial\Omega\cap Q(0, 2R)} |\psi|^2\, d\sigma
-\int_{|x^\prime|< R} |\varphi|^2\, dx\\
&=\int_{|x^\prime|< R}
|\psi (x^\prime, \phi (x^\prime))|^2 
\big\{ \sqrt{1+|\nabla \phi|^2 } -1 \big\} dx^\prime\\
&\le M_R^2  \int_{|x^\prime|< R}
|\varphi |^2  dx^\prime.
\endaligned
$$
This implies that 
$$
\aligned
\mu^{D\!N} (\A, \Omega
\cap Q(0, 2R), \Omega)
 & \le (1+ CM_R R^{\kappa+1}) (1+ C M_R^2)
\mu^{D\!N}(\A, \R^d_+\cap \Phi(Q(0, R)), \R^d_+)\\
&\le  \mu^{D\!N}(\A, \R^d_+\cap \Phi(Q(0, R), \R^d_+)
+ C M_R R^{\kappa+1}\\
&\le \mu^{D\!N} (\A, \R^d_+ \cap Q (0, cR), \R^d_+)
+C M_R R^{\kappa+1}\\
&\le \lambda^{D\!N} (\A, \R^d_+ )+ C R^{-2} 
+C M_R R^{\kappa+1},
\endaligned
$$
where we have used the  fact $Q(0, cR) \subset \Phi(Q(0, R))$ for the third inequality and
Lemma \ref{lemma-p4} for the fourth.
The lower bound for $\mu^{D\!N} (\A, \Omega\cap Q(0, 2R), \Omega)$ may be proved by a similar argument.
\end{proof}


\section{Local asymptotic expansions }\label{section-local}

In this section we establish the error estimates for \eqref{est-1}-\eqref{est-2}.
For $y\in \overline{\Omega}$, let $\kappa =\kappa(y)$ be defined by \eqref{ex-1}.
Let $
\P_{y} (x)= (P_{j \ell} (x))=\sum_{|\alpha|=\kappa}  b_\alpha x^\alpha$, where
$b_\alpha =\partial^\alpha \B (y)/\alpha!$, denote the $\kappa^{th}$ Taylor polynomial of 
$\B(x+y)$ at $0$.
Let $\A_{y} =(A_{y, 1}, \dots, A_{y, d} )$ be  the homogeneous polynomial of degree $\kappa +1$, given by 
\begin{equation}\label{A-0}
A_{y, j} (x)=\sum_{\ell=1}^d x_\ell \int_0^1 P_{ \ell j } (tx) t\,  dt
\end{equation}
for $1\le j\le d$.
Then $\nabla \times \A_{y} = \P_{y}$.
Let $V_y$ denote the invariance subspace for $\P_y$.
Since $\nabla \partial^\alpha P_{j\ell } (0) =\nabla \partial^\alpha B_{j \ell}  (y)$ for $|\alpha |\le \kappa-1$, in view of Proposition \ref{prop-1}, we have
\begin{equation}\label{sigma-1a}
V_y =\Big \{ z\in \R^d: 
 \sum_{|\alpha|=\kappa-1} \sum_{j, \ell}  | \langle z, \nabla \partial^\alpha B_{j \ell}  (y) \rangle | = 0 \Big\}.
\end{equation}
Also, by \eqref{mp-1b}, 
\begin{equation}\label{sigma-1}
\sigma (y)=\min_{{\substack{z\in V_y^\perp\\  |z| =1}}}
\sum_{|\alpha|= \kappa-1 } \sum_{j,  \ell} 
| \langle z, \nabla \partial^\alpha B_{j \ell} (y) \rangle |.
\end{equation}

Let $\widetilde{\Gamma}=\{ y \in \widetilde{\Omega}: \kappa (y)=\kappa_*\}$, where $\kappa_*$ is defined by 
\eqref{kappa} and $\widetilde{\Omega}=\{ x\in \R^d: \text{dist}(x, \Omega)< r_0 \}$.
Let $y\in \widetilde{\Gamma}$.
If  $\gamma (t): (-c, c) \to \widetilde{\Gamma}$ is a $C^1$ function such that $\gamma (0)=y$, 
 then  $\partial^\alpha B_{j\ell} (\gamma (t)) =0$ for $|\alpha|=\kappa_*-1$.
 It follows that $\langle \gamma^\prime (0), \nabla \partial^\alpha B_{j\ell} (y) \rangle =0$.
Thus,  $\gamma^\prime (0) \in V_y$.
This shows that $V_y$ contains the tangent space for $\widetilde{\Gamma}$ at $y$ if $\widetilde{\Gamma}$ is a smooth manifold near $y$. 

\begin{thm}\label{thm-local-1}
Let $y\in \Omega$ and
 $r=\gamma  \beta^{-\frac{\kappa+3}{(\kappa+2)(\kappa+4)}}$, where $\beta>1$, $ 0< \gamma\le  1$  and $\kappa =\kappa (y)$. Suppose $\B(y, r)\subset \Omega$. Then 
\begin{equation}\label{AE-I}
|\lambda^D(\beta \A, \b(y, r))-\beta^{\frac{2}{\kappa+2} } \lambda (\A_{y}, \R^d)|
\le C \beta^{\frac{1}{\kappa+2} +\frac{1}{\kappa+4}},
\end{equation}
where $C$ depends on  $\gamma$,  $\kappa$, $\| \B \|_{C^{\kappa+1}(\overline{\Omega})}$ and $\sigma (y)$.
\end{thm}

\begin{proof}

By translation we may assume $y=0$.
Suppose  $\b(0, r) \subset \Omega$.
Let $\P (x)$ denote the $\kappa^{th}$ Taylor polynomial of $\B$ at $0$.
Up to a gauge transformation, we may assume that for $x\in \b(0, r)$,
$\A (x)  =\A_{0}  (x)  + \mathcal{R} (x) $, 
where $\nabla\times \A_{0} =\P$,  
\begin{equation}\label{up2-1-a}
\aligned
\A_{0} (x)  & =\sum_{|\alpha|=\kappa+1}
a_\alpha x ^\alpha, \\
|\mathcal{R} (x)|  & \le C |x|^{\kappa+2},
\endaligned
\end{equation}
and  $C$ depends on $\| \B \|_{C^{\kappa+1}(\overline{\Omega})}$.
Using
$$
\aligned
& \left(\fint_{\b(0, r)} |(D+\beta \A)\psi|^2 \right)^{1/2}\\
&\le 
\left(\fint_{\b(0, r)} |(D+\beta \A_0)\psi|^2 \right)^{1/2}
+ \beta \max_{\b(0, r)} | \A-\A_0|
\left(\fint_{\b(0, r)} |\psi|^2 \right)^{1/2},
\endaligned
$$
we obtain 
$$
\sqrt{ \lambda^D(\beta\A, \b(0, r))}
\le \sqrt{\lambda^D (\beta\A_0, \b(0,  r))}
+\beta \max_{\b(0, r)} |\A-\A_0|.
$$
Hence, by \eqref{up2-1-a}
$$
\lambda^D (\beta\A, \b (0, r))
\le \lambda^D (\beta \A_0, \b (0, r))
+ C \beta r^{\kappa+2} \sqrt{\lambda^D(\beta \A_0, \b(0, r))}
+ C \beta^2 r^{2 (\kappa+2)}.
$$
A similar argument gives\
$$
\aligned
\lambda^D (\beta\A_0, \b (0, r))
 & \le \lambda^D (\beta \A, \b (0, r))
+ C \beta r^{\kappa+2} \sqrt{\lambda^D(\beta \A, \b(0, r))}
+ C \beta^2 r^{2 (\kappa+2)}.\\
&  \le \lambda^D (\beta \A, \b (0, r))
+ C \beta r^{\kappa+2} \sqrt{\lambda^D(\beta \A_0, \b(0, r))}
+ C \beta^2 r^{2 (\kappa+2)}.
\endaligned
$$
Thus, 
\begin{equation}\label{local-1a}
\aligned
 | \lambda^D (\beta\A, \b (0, r)) - \lambda^D (\beta \A_0, \b (0, r) |
\le  C \beta r^{\kappa+2} \sqrt{\lambda^D(\beta \A_0, \b(0, r))}
+ C \beta^2 r^{2 (\kappa+2)}.\\
\endaligned
\end{equation}
Let $r= R \beta^{-\frac{1}{\kappa+2}}$.
Since $\A_0 (tx)= t^{\kappa+1} \A_0(x)$, 
by a rescaling argument, it follows that
\begin{equation}
\lambda^D(\beta \A_0, \b(0, r))
=\beta^{\frac{2}{\kappa+2}}
\lambda^D (\A_0, \b (0, R)).
\end{equation}
As a result,
\begin{equation}
 | \lambda^D(\beta \A, \b(0, r))
- 
\beta^{\frac{2}{\kappa+2}}
\lambda^D (\A_0, \b (0, R)) |
\le  C \beta^{\frac{1}{\kappa+2}} R^{\kappa+2}
+ C R^{2(\kappa+2)},
\end{equation}
where we have used the estimate $\lambda^D (\A_0, \b(0, R))\le C$ for $R>1$.
In view of Lemma \ref{lemma-p1}, this gives
\begin{equation}\label{up1ab}
 | \lambda^D(\beta \A, \b(0, r))- 
\beta^{\frac{2}{\kappa+2}}
\lambda (\A_0, \R^d)|
\le  C \beta^{\frac{2}{\kappa+2}} R^{-2} 
+ C \beta^{\frac{1}{\kappa+2}} R^{\kappa+2}
+ C R^{2(\kappa+2)},
\end{equation}
where $C$ depends on $\sigma(y)$.
Finally,  choosing $R=\gamma \beta^{\frac{1}{(\kappa+2)(\kappa+4)}}$ to optimize \eqref{up1ab}, we see that 
$r=R\beta^{-\frac{1}{\kappa+2}}=\gamma \beta^{-\frac{\kappa+3}{(\kappa+2)(\kappa+4)}}$ and 
\begin{equation}\nonumber
\aligned
 | \lambda^D (\beta \A, \b (0, r))
 -  \beta^{\frac{2}{\kappa+2}}
\lambda (\A_0, \R^d)|
\le  C \beta^{\frac{1}{\kappa+2} +\frac{1}{\kappa+4}},
\endaligned
\end{equation}
where $C$ depend on $\gamma$,  $\sigma (y)$ and $\|\B \|_{C^{\kappa+1}(\overline{\Omega})}$.
\end{proof}

For $n \in \mathbb{S}^{d-1}$, let
$$
\mathbb{H}_n
=\big\{ x\in \R^d: \  \langle x, n \rangle < 0 \big\}
$$
denote the half-space  with outer normal $n$.
In particular, if $n=(0, \dots, 0, -1)$, we have $\mathbb{H}_n =\R^d_+$.
For $y \in \partial\Omega$, let $n(y)$ denote the outward unit normal to $\partial\Omega$ at $y$ and 
\begin{equation}\label{tau}
\tau (y)=\max_{v\in V_y} | \langle v, n(y) \rangle |, 
\end{equation}
where $V_y$ is given by \eqref{sigma-1a}.

\begin{thm}\label{thm-local-2}
Let $\Omega$ be a bounded $C^{1, 1}$ domain.
Let $y\in \partial  \Omega$ and
 $r=\gamma \beta^{-\frac{\kappa+3}{(\kappa+2)(\kappa+4)}}$, where $\beta>1$ and $0< \gamma\le 1$. 
 Let $\kappa=\kappa (y)$ and $n =n (y)$.
 Then 
\begin{equation}\label{AE-D}
|\lambda^D(\beta \A, \b(y, r)\cap \Omega)-\beta^{\frac{2}{\kappa+2} } \lambda ^D(\A_{y}, \mathbb{H}_n)|
\le C \beta^{\frac{1}{\kappa+2} +\frac{1}{\kappa+4}},
\end{equation}
where $C$ depends on $\gamma$,  $\Omega$,  $\kappa$, $\|\B \|_{C^{\kappa+1}(\overline{\Omega})}$, $\sigma (y)$ and  $\tau(y)>0$.
\end{thm}

\begin{proof}

By translation and rotation, we may assume that $y=0$, $n=(0, \dots, 0, -1)$, and
\begin{equation}\label{o-l}
\Omega \cap Q(0, 2r_0)
=\big\{ (x^\prime, x_d)\in \R^d: \ x_d > \phi (x^\prime) \big\} \cap Q(0, 2r_0),
\end{equation}
where  $\phi: \R^{d-1} \to \R$ is a $C^{1, 1}$ function with $\phi (0)=0$ and $\nabla \phi (0)=0$.
We may assume $\beta$ is so large that $0< r< r_0$.
As in the proof of the previous theorem,
$$
\aligned
& |  \lambda^D (\beta\A, Q (0, r)\cap \Omega)
  -\lambda^D (\beta \A_0, Q (0, r)\cap \Omega) | \\
 &\qquad \le C \beta r^{\kappa+2} \sqrt{\lambda^D(\beta \A_0, Q(0, r)\cap \Omega)}
+ C \beta^2 r^{2 (\kappa+2)}
\endaligned
$$
for $0< r< r_0$.
Let $r= R\beta^{-\frac{1}{\kappa+2}}< r_0$, where $R>1$.
By rescaling,
\begin{equation}
\lambda^D(\beta\A_0, Q(0, r)\cap \Omega)
=\beta^{\frac{2}{\kappa+2}}
\lambda^D (\A_0, Q (0, R)\cap \Omega_\beta ),
\end{equation}
where
\begin{equation}\label{o-b}
\Omega_\beta
=\big\{ (x^\prime, x_d)\in \R^d:\
x_d > \phi_\beta (x^\prime)  \big\}
\end{equation}
and 
$$
\phi_\beta (x^\prime) =\beta^{\frac{1}{\kappa +2}} \phi (\beta^{-\frac{1}{\kappa +2}} x^\prime).
$$
It follows that
\begin{equation}\label{up2-2a}
\aligned
|  \lambda^D (\beta \A; Q (0, r) \cap\Omega)
 - \beta^{\frac{2}{\kappa+2}}
\lambda^D (\A_0, Q (0, R)\cap \Omega_\beta) |
\le C \beta^{\frac{1}{\kappa+2}} R^{\kappa+2}
+ C R^{2(\kappa+2)},
\endaligned
\end{equation}
where $1< R< r_0 \beta^{\frac{1}{\kappa+2}}$.
Note that if $|x^\prime|< R$, 
\begin{equation}\label{up2-2b}
|\nabla \phi_\beta (x^\prime)|
=|\nabla \phi (\beta^{-\frac{1}{\kappa +2}}x^\prime ) |
\le C R \beta^{-\frac{1}{\kappa +2}}, 
\end{equation}
where we have used the assumption that $\nabla \phi (0)=0$ and  $\phi$ is $C^{1, 1}$.
It follows by Theorem \ref{thm-p5}  and the estimate  \eqref{up2-2b} that
$$
 | \lambda ^D (\A_0, \Omega_\beta \cap Q(0, R))
 - \lambda^D (\A_0, \R^d_+)  |
\le  C R^{\kappa+2} \beta^{-\frac{1}{\kappa +2}}
+ C R^{-2},
$$
where $C$ depends on $\sigma (y)$ and $\tau (y)$.
This, together with \eqref{up2-2a}, gives
$$
 | \lambda ^D (\beta \A,  Q(0, r) \cap \Omega)
- \beta^{\frac{2}{\kappa+2}} 
\lambda^D (\A_0, \R^d_+) |
\le C R^{-2} \beta^{\frac{2}{\kappa+2}} 
+C \beta^{\frac{1}{\kappa+2}} R^{\kappa+2}
+ C R^{2(\kappa+2)}.
$$
Finally, by choosing $R=\gamma \beta^{\frac{1}{(\kappa+2)(\kappa+4)}}$, we obtain \eqref{AE-D}.
\end{proof}

\begin{thm}\label{thm-local-3}
Let $\Omega$ be a bounded $C^{1, 1}$ domain.
Let $y\in \partial  \Omega$ and
 $r=\gamma \beta^{-\frac{\kappa+3}{(\kappa+2)(\kappa+4)}}$, where $\beta>1$ and $0< \gamma\le 1$. 
 Let $\kappa=\kappa (y)$ and $n =n (y)$.
 Then 
\begin{equation}\label{AE-N}
|\mu^N(\beta \A, \b(y, r)\cap \Omega, \Omega)-\beta^{\frac{2}{\kappa+2} } \lambda ^N(\A_{y}, \mathbb{H}_n)|
\le C \beta^{\frac{1}{\kappa+2} +\frac{1}{\kappa+4}},
\end{equation}
where $C$ depends on $\gamma$,  $\Omega$,  $\kappa$, $\|\B \|_{C^{\kappa+1}(\overline{\Omega})}$, $\sigma (y)$ and $\tau (y)$.
\end{thm}

\begin{proof}

The proof is similar to that of Lemma \ref{thm-local-2}, using Theorem \ref{thm-p6},
\end{proof}

\begin{thm}\label{thm-local-4}
Let $\Omega$ be a bounded $C^{1, 1}$ domain.
Let $y\in \partial  \Omega$ and
 $r=\gamma \beta^{-\frac{\kappa+3}{(\kappa+2)(\kappa+4)}}$, where $\beta>1$ and $0< \gamma \le 1$. 
 Let $\kappa=\kappa (y)$ and $n =n (y)$.
 Then 
\begin{equation}\label{AE-DN}
|\mu^{D\!N} (\beta \A, \b(y, r)\cap \Omega, \Omega)-\beta^{\frac{1}{\kappa+2} } \lambda ^{D\!N}(\A_{y}, \mathbb{H}_n)|
\le C \beta^{\frac{1}{\kappa+4}}, 
\end{equation}
where $C$ depends on $\gamma$,   $\Omega$,  $\kappa$, $\|\B \|_{C^{\kappa+1}(\overline{\Omega})}$, $\sigma (y)$ and $\tau (y)$.
\end{thm}

\begin{proof}

By translation and rotation we may assume that $y=0$,  $n=(0, \dots, 0, -1)$, and 
$\Omega \cap Q(0, 2r_0)$ is given by \eqref{o-l}.
Let $r=R \beta^{-\frac{1}{\kappa +2}}< r_0$, where $R>1$.
Assume that $R^{\kappa +2} \beta^{-\frac{1}{\kappa +2}}<<1$.
Note that for $\psi \in C^1(Q(0, r)\cap \Omega; \C)$ such that $\psi =0$ on $\Omega\cap \partial Q(0, r)$, 
$$
\aligned
 & \left(\int_{Q(0, r)\cap \Omega}
|(D+\beta \A)\psi|^2\right)^{1/2}\\
&\le 
\left(\int_{Q(0, r)\cap \Omega}
|(D+\beta \A_0)\psi|^2\right)^{1/2}
+ \beta \max_{Q(0, r)} |\A -\A_0|
\left(\int_{Q(0, r)\cap \Omega}
|\psi|^2 \right)^{1/2}\\
&\le \left( 1 + C \beta r^{\kappa+2} [ \mu^N (\beta \A_0, Q(0, r)\cap \Omega, \Omega) ]^{-1/2} \right)
\left(\int_{Q(0, r)\cap \Omega}
|(D+\beta \A_0)\psi|^2\right)^{1/2}\\
 &\le \left( 1 + C  R^{\kappa+2} \beta^{-\frac{1}{\kappa+2}}  \right)
\left(\int_{Q(0, r)\cap \Omega}
|(D+\beta \A_0)\psi|^2\right)^{1/2},\\
\endaligned
$$
where we have used the fact 
$$
\mu^N(\beta \A_0, Q(0, r)\cap \Omega, \Omega)
=\beta^{\frac{2}{\kappa+2}}
\mu^N(\A_0, Q(0, R) \cap \Omega_\beta, \Omega_\beta) 
\ge c \beta^{\frac{2}{\kappa +2}},
$$
with $\Omega_\beta$ given by \eqref{o-b}.
A similar argument gives
$$
 \left(\int_{Q(0, r)\cap \Omega}
|(D+\beta \A_0)\psi|^2\right)^{1/2}
\le \left( 1 - C  R^{\kappa+2} \beta^{-\frac{1}{\kappa+2}}  \right)^{-1} 
\left(\int_{Q(0, r)\cap \Omega}
|(D+\beta \A)\psi|^2\right)^{1/2}.
$$
 It follows that
\begin{equation}\label{loc-101}
\aligned
& \mu^{D\!N}(\beta\A, Q(0, r) \cap \Omega, \Omega)\\
& \le  \left( 1 + C  R^{\kappa+2} \beta^{-\frac{1}{\kappa+2}}   \right)^2
\mu^{D\!N}(\beta\A_0, Q(0, r) \cap \Omega, \Omega).\\
&=  \left( 1 + C  R^{\kappa+2} \beta^{-\frac{1}{\kappa+2}}   \right)
\beta^{\frac{1}{\kappa+2}}
\mu^{D\!N} (\A_0, Q(0, R)\cap \Omega_\beta, \Omega_\beta).
\endaligned
\end{equation}
and
\begin{equation}
\beta^{\frac{1}{\kappa+2}}
\mu^{D\!N} (\A_0, Q(0, R)\cap \Omega_\beta, \Omega_\beta)
\le   \left( 1 + C  R^{\kappa+2} \beta^{-\frac{1}{\kappa+2}}   \right)
 \mu^{D\!N}(\beta\A, Q(0, r) \cap \Omega, \Omega).
\end{equation}
This, together with  Theorem \ref{thm-p7}, yields 
$$
\aligned
 &  \mu^{D\!N}(\beta\A, Q(0, r) \cap \Omega, \Omega)\\
&\qquad\le \beta^{\frac{1}{\kappa+2}}
\left\{ \lambda^{D\!N} (\A_0, \R^d_+)
+ C R^{\kappa+2} \beta^{-\frac{1}{\kappa+2}}
+ CR^{-2} \right\}
  \left\{ 1+ C R^{\kappa+2} \beta^{-\frac{1}{\kappa+2}} \right\}\\
&\qquad  \le \beta^{\frac{1}{\kappa+2} }\lambda^{D\!N} (\A_0, \R^d_+)
+C R^{\kappa+2} + C R^{-2} \beta^{\frac{1}{\kappa+2}}, 
\endaligned
$$
and
$$
\aligned
 \beta^{\frac{1}{\kappa+2} }\lambda^{D\!N} (\A_0, \R^d_+)
 &\le \beta^{\frac{1}{\kappa+2}} \mu^{D\!N} (A_0, Q(0, R)\cap \Omega_\beta, \Omega_\beta)
 +C R^{\kappa+2}  + CR^{-2} \beta^{\frac{1}{\kappa+2}}\\
 & \le   \mu^{D\!N}(\beta\A, Q(0, r) \cap \Omega, \Omega)
 + C R^{\kappa+2}  + CR^{-2} \beta^{\frac{1}{\kappa+2}}.
\endaligned
$$
By choosing $R=\gamma \beta^{\frac{1}{(\kappa+2) (\kappa+4)}}$, we obtain  \eqref{AE-DN}.
\end{proof}


\section{Upper bounds, part II}

 Let $\Omega$ be a bounded Lipschitz domain.
 Let $\kappa_*$ and $\kappa_0$ be defined by \eqref{kappa} and \eqref{kappa-1}, respectively.
 It follows  by Theorem \ref{main-thm-1}  that if   $|\B|$ does not vanish to infinite order at any point in $\overline{\Omega}$, then
$$
\lambda^D(\beta \A, \Omega), \lambda^N (\beta\A, \Omega)
\approx  \beta^{\frac{2}{\kappa_* +2}} 
$$
for $\beta>1$ large. Also, by Theorem \ref{main-thm-2},  if  $|\B|$ does not vanish to infinite order at any point on $\partial \Omega$,  then
$$
\lambda^{D\!N} (\beta \A, \Omega)
\approx \beta^{\frac{1}{k_0 +2}}
$$
for $\beta>1$ large.
In this section we derive more precise upper bounds for 
$$
\beta^{-\frac{2}{\kappa_*+2}} \lambda^D (\beta \A, \Omega), \quad 
\beta^{-\frac{2}{\kappa_*+2}} \lambda^N (\beta \A, \Omega), \quad \text{ and } \quad 
\beta^{-\frac{1}{\kappa_0+2}} \lambda^{D\!N} (\beta \A, \Omega),
$$
 as $\beta \to \infty$.
 
\begin{lemma}\label{lemma-up2-1}
Let $\Omega$ be a bounded Lipschitz domain in $\R^d$.
Let $y\in \Omega$ and $\kappa=\kappa(y)$.
Then 
\begin{equation}\label{up2-1-0}
\lambda^N (\beta \A, \Omega) \le \lambda^D (\beta\A, \Omega)
\le\  \beta^{\frac{2}{\kappa+2}}
\lambda (\A_{y}, \R^d)
+ C \beta^{\frac{1}{\kappa+2} +\frac{1}{\kappa+4}}
\end{equation}
for $\beta>\beta_{y}$, where $\A_{y} $  is   a homogeneous  polynomial  of degree $\kappa+1$ given by \eqref{A-0}.
The constant $C$ depends on $\kappa$, $\|\B\|_{C^{\kappa+1}(\overline{\Omega})}$ and $\sigma (y)$ in \eqref{sigma-1}.
\end{lemma}

\begin{proof}

Let $r=\beta^{-\frac{\kappa+3}{(\kappa+2)(\kappa+4)}}$.
Suppose $\b(y, r)\subset \Omega$. Then,
$$
\aligned
\lambda^N(\beta \A, \Omega)  & \le 
\lambda^D (\beta \A, \Omega)\\
& \le \lambda^D (\beta \A, \b(y, r) )
\le \beta^{\frac{2}{\kappa+2}} \lambda (\A_y, \R^d) + C \beta^{\frac{1}{\kappa+2} +\frac{1}{\kappa+4}},
\endaligned
$$
where we have used \eqref{AE-I} for the last inequality.
\end{proof}

\begin{lemma}\label{lemma-up2-2}
Let $\Omega$ be a bounded $C^{1, 1}$ domain in $\R^d$.
Let $n=n(y)$ denote the outward unit normal to $\partial\Omega$ at $y\in \partial\Omega$ and 
$\kappa = \kappa(y)$.
Then 
\begin{equation}\label{up2-2-0}
\lambda^D (\beta\A, \Omega)
\le\  \beta^{\frac{2}{\kappa+2}}
\lambda^D (\A_{y},  \mathbb{H}_n)
+ C \beta^{\frac{1}{\kappa+2} +\frac{1}{\kappa+4}},
\end{equation}
\begin{equation}\label{up2-3-0}
\lambda^N (\beta\A, \Omega)
\le\  \beta^{\frac{2}{\kappa+2}}
\lambda^N (\A_{y},  \mathbb{H}_n )
+ C \beta^{\frac{1}{\kappa+2} +\frac{1}{\kappa+4}}, 
\end{equation}
for $\beta> 1$, where $\A_{y}$ is a homogeneous  polynomial  of degree $\kappa+1$ given by \eqref{A-0}.
The constant $C$ depends on $\kappa$, $\|\B\|_{C^{\kappa+1}(\overline{\Omega})}$, $\sigma (y)$ and $\tau (y)$ in \eqref{tau}.
\end{lemma}

\begin{proof}

Since $\lambda^D(\beta \A, \Omega) \le \lambda^D (\beta\A, \b(y, r)\cap \Omega)$
and $\lambda^N (\beta\A, \Omega) \le \mu^N (\beta\A, \b(y, r) \cap \Omega,  \Omega)$,
 where $r=c\beta^{-\frac{\kappa+3}{(\kappa+2)(\kappa+4)}}<r_0$,
this follows readily from \eqref{AE-D} and \eqref{AE-N}.
\end{proof}

\begin{thm}\label{thm-up2a}
Suppose that $\A\in C^\infty(\R^d; \R^d)$.
Let $\Omega$ be a bounded $C^1$ domain.
Assume that $|\B|$ does not vanish to infinite order
at any point in $\overline{\Omega}$.
Let $\kappa_*$ be defined by \eqref{kappa}.
Then
\begin{equation}\label{up2-3a}
\aligned
 & \limsup_{\beta\to \infty} 
\beta^{-\frac{2}{\kappa_* +2}}
\lambda^D (\beta \A, \Omega)
\le \Theta_D, 
\endaligned
\end{equation}
\begin{equation}\label{up2-3b}
\aligned
  \limsup_{\beta\to \infty} 
\beta^{-\frac{2}{\kappa_* +2}}
\lambda^N (\beta \A, \Omega)
 \le \Theta_N, 
\endaligned
\end{equation}
where $\Theta_D$  and $ \Theta_N$ are  given by \eqref{co-1}.
\end{thm}

\begin{proof}

We give the proof for \eqref{up2-3a}.
The proof for \eqref{up2-3b} is similar.
If $\Omega$ is $C^{1, 1}$, the inequality \eqref{up2-3a} follows readily from Lemmas \ref{lemma-up2-1} and \ref{lemma-up2-2}.
If $\Omega$ is $C^1$, we first use  Lemma \ref{lemma-up2-1} to obtain 
\begin{equation}\label{up2a-1}
 \limsup_{\beta\to \infty} 
\beta^{-\frac{2}{\kappa_* +2}}
\lambda^D (\beta \A, \Omega)
\le
\inf_{y\in \Gamma_1} \lambda(\A_{y}, \R^d).
\end{equation}
To complete the proof of  \eqref{up2-3a}, we use the proof of Lemma \ref{lemma-up2-2}.
Let $y\in \Gamma_2$.
Without the loss of generality, we may assume $y=0$, $n(y)=(0, \dots, 0, -1)$ and \eqref{o-l} holds.
Since $\Omega$ is $C^1$, we have 
$$
|\nabla \phi_\beta (x^\prime)| = o(R\beta^{-\frac{1}{\kappa_* +2}}) \qquad \text{ as } \beta\to \infty.
$$
By Theorem \ref{thm-p5}, this gives 
$$
\lambda^D (\A_0, \Omega_\beta \cap Q(0, R)) \le \lambda^D (\A_0, \R^d_+) + R^{\kappa_*+1} o(R \beta^{-\frac{1}{\kappa_*+2}})
+ C R^{-2},
$$
where $1<R< r_0\beta^{\frac{1}{\kappa_*+2}}$.
It then follows from \eqref{up2-2a} that
$$
\aligned
\beta^{-\frac{2}{\kappa_*+2}} \lambda^D (\beta \A, \Omega)
&\le \lambda^D (\A_0, Q(0, R)\cap \Omega_\beta)
+ C \beta^{-\frac{1}{\kappa_*+2}} R^{\kappa_*+2}
+C \beta^{-\frac{2}{\kappa_*+2}} R^{2(\kappa_*+2)}\\
&\le  \lambda^D (\A_0, \R^d_+) + R^{\kappa_*+1} o(R \beta^{-\frac{1}{\kappa_*+2}})
+ C R^{-2}\\
&\qquad\qquad
+ C \beta^{-\frac{1}{\kappa_*+2}} R^{\kappa_*+2}
+C \beta^{-\frac{2}{\kappa_*+2}} R^{2(\kappa_*+2)}.\\
\endaligned
$$
Hence, for  any $R>1$,
$$
\limsup_{\beta \to \infty} \beta^{-\frac{2}{\kappa_*+2}}
\lambda^D (\beta \A, \Omega)
\le    \lambda^D (\A_0, \R^d_+) 
+ C R^{-2}.
$$
By letting $R\to \infty$, we obtain 
$$
\limsup_{\beta \to \infty} \beta^{-\frac{2}{\kappa_*+2}}
\lambda^D (\beta \A, \Omega)
\le    \lambda^D (\A_0, \R^d_+), 
$$
which, together with \eqref{up2a-1}, yields \eqref{up2-3a}.
\end{proof}

\begin{thm}\label{thm-mup1}
Suppose that $\A\in C^\infty(\R^d; \R^d)$.
Let $\Omega$ be a bounded $C^{1, 1} $ domain.
Assume that $|\B|$ does not vanish to infinite order
at any point in $\overline{\Omega}$.
Let $\kappa_*$ be defined by \eqref{kappa} and
\begin{equation}\label{W}
\aligned
\Gamma_1 & = \left\{ y \in \Omega: \ \kappa (y) = \kappa_* \right\},\\
\Gamma_2 & = \left\{ y \in \partial \Omega: \ \kappa (y) = \kappa_* \right\}.
\endaligned
\end{equation}
Suppose that there exists $c>0$ such that 
\begin{equation}\label{cond-1g}
\min_{v\in \mathbb{S}^{d-1}\cap V_y^\perp}
\sum_{j, \ell} \sum_{|\alpha|=\kappa_*-1}
|\langle v, \nabla \partial^\alpha B_{j \ell} (y) \rangle |\ge c
\end{equation}
 for  $y\in \Gamma_*=\Gamma_1\cup \Gamma_2$, and that 
\begin{equation}\label{cond-2g}
\max_{v\in V_y} |\langle v, n(y) \rangle | \ge c \qquad
\end{equation}
 for  $ y \in \Gamma_2$, where $V_y$ is given by \eqref{sigma-1a}.
 We further assume that  for any $y_0\in \overline{\Gamma}_1\cap\partial \Omega$ and $0< r< cr_0$,
 there exists $y_r\in \Gamma_1$ such that 
\begin{equation}\label{cond-3g}
\b(y_r, cr) \subset \Omega \quad \text{ and } \quad
|y_r-y_0|\le C r.
\end{equation}
Then
\begin{equation}\label{up2-3aa}
\lambda^D (\beta \A, \Omega)
\le \Theta_D \beta^{\frac{2}{\kappa_*+2}} 
+ C \beta^{\frac{1}{\kappa_*+2} + \frac{1}{\kappa_*+4}},
\end{equation}
\begin{equation}\label{up2-3ab}
\lambda^N (\beta \A, \Omega)
\le \Theta_N \beta^{\frac{2}{\kappa_*+2}} 
+ C \beta^{\frac{1}{\kappa_*+2} + \frac{1}{\kappa_*+4}},
\end{equation}
for $\beta$ large, 
where $\Theta_D$ and $\Theta_N$ are  given by  \eqref{co-1}.
\end{thm}

\begin{proof}

Under the conditions \eqref{cond-1g}-\eqref{cond-2g},
it follows by Lemma \ref{lemma-up2-2} that \eqref{up2-2-0} holds uniformly in $y \in \Gamma_2$. Hence, 
$$
\lambda^D (\beta\A, \Omega)
\le\  \beta^{\frac{2}{\kappa_*+2}}
\inf_{y\in \Gamma_2}\lambda^D (\A_{y},  \mathbb{H}_{n(y)} )
+ C \beta^{\frac{1}{\kappa_*+2} +\frac{1}{\kappa_*+4}}
$$
for $\beta$ large. This gives \eqref{up2-3aa} in the case
$\Theta_D =\inf_{y\in \Gamma_2}\lambda^D (\A_{y},  \mathbb{H}_{n(y)} )$.

Next, suppose that 
\begin{equation}\label{U-C}
\Theta_D=\inf_{y\in \Gamma_1} \lambda(\A_y, \R^d) < \inf_{y\in \Gamma_2}\lambda^D (\A_{y},  \mathbb{H}_{n(y)} ).
\end{equation}
To see \eqref{up2-3aa}, 
we first note that under the condition \eqref{cond-1g}, 
\begin{equation}\label{A-C}
|\lambda(\A_{y_1}, \R^d) -\lambda (\A_{y_2}, \R^d)|
\le C | y_1-y_2|^{\frac{2}{\kappa_*+3}}
\end{equation}
for any $y_1, y_2\in \Gamma_*$.
Indeed,  by Lemma \ref{lemma-p1},
$$
\aligned
|\lambda(\A_{y_1}, \R^d) -\lambda (\A_{y_2}, \R^d)|
&\le |\lambda(\A_{y_1}, \b(0, R)) -\lambda (\A_{y_2}, \b(0, R))| + C R^{-2}\\
& \le C \| \A_{y_1} -\A_{y_2} \|_{L^\infty(\b(0, R))} + C R^{-2}\\
&\le C |y_1-y_2| R^{\kappa_*+1} + C R^{-2}
\endaligned
$$
for $R>1$, where $C$ depends on $\kappa_*$, $\| \B \|_{C^{\kappa_*+1}(\overline{\Omega})}$ and $c$ in \eqref{cond-2g}.
We obtain \eqref{A-C} by choosing $R=|y_1-y_2|^{-\frac{1}{\kappa_*+3}}$.
It follows from \eqref{A-C} that $\lambda(\A_y, \R^d)$ as a function of $y$ is uniformly continuous in $\Gamma_1$.
As a result, 
$$
\Theta_D=\min_{y\in \overline{\Gamma}_1} \lambda(\A_y, \R^d)
=\lambda(\A_{y_0}, \R^d)
$$
for some $y_0 \in \overline{\Gamma}_1$.
Note that if $y_0\in \Gamma_1$, then \eqref{up2-3aa} follows directly from \eqref{up2-1-0}.

Finally, suppose $y_0\in \overline{\Gamma}_1 \setminus \Gamma_1$.  Then $y_0\in \overline{\Gamma}_1 \cap \partial\Omega$.
For $ \beta>1$, let $r=c \beta^{-\frac{\kappa_*+3}{(\kappa_*+2)(\kappa_*+4)}} < r_0$.
Choose $y_r\in \Gamma_1$ such that $|y_r -y_0|\le C r$ and $\b(y_r, cr)\subset \Omega$.
Then
$$
\aligned
\lambda^D (\beta \A, \Omega)
& \le \lambda^D (\beta \A, \b(y_r, cr))\\
& \le \beta^{\frac{2}{\kappa_*+2}}  
\lambda (\A_{y_r}, \R^d) + C \beta^{\frac{1}{\kappa_*+2} +\frac{1}{\kappa_*+4}}\\
& \le  \beta^{\frac{2}{\kappa_*+2}}  
\lambda (\A_{y_0}, \R^d) +  C \beta^{\frac{2}{\kappa_*+2}} |y_1 -y_0|^{\frac{2}{\kappa_*+3}}
+C \beta^{\frac{1}{\kappa_*+2} +\frac{1}{\kappa_*+4}}\\
&\le \beta^{\frac{2}{\kappa_*+2}}  
\Theta_{D}  +\beta^{ \frac{1}{\kappa_*+4}}, 
\endaligned
$$
where we have used \eqref{A-C} for the third inequality.
\end{proof}

\begin{lemma}\label{lemma-up2-4}
Let $\Omega$ be a bounded $C^{1, 1}$ domain in $\R^d$. For $y\in \partial\Omega$, 
let  $n=n(y)$ denote the outward unit normal to $\partial\Omega$ at $y$ and 
$\kappa = \kappa(y)$.
Then 
\begin{equation}\label{up2-4-0}
\lambda^{D\!N} (\beta\A, \Omega)
\le\  \beta^{\frac{1}{\kappa+2}}
\lambda^{D\!N} (\A_{y},  \mathbb{H}_n )
+ C \beta^{\frac{1}{\kappa+4}}
\end{equation}
for $\beta>1$, where $\A_{y}$ is a polynomial  of degree $\kappa+1$ given by \eqref{A-0}.
The constant $C$ depends on $\kappa$, $\|\B\|_{C^{\kappa+1}(\overline{\Omega})}$, $\sigma (y)$ and $\tau (y)$ in \eqref{tau}.
\end{lemma}

\begin{proof}

Since $\lambda^{D\!N} (\beta \A, \Omega)
\le \mu^{D\!N} (\beta \A, \b(y, r) \cap \Omega, \Omega)$ for $y\in \partial\Omega$ and $0< r< r_0$,
the inequality \eqref{up2-4-0} follows readily from \eqref{AE-DN}.
\end{proof}

\begin{thm}\label{thm-up5a}
Suppose that $\A\in C^\infty(\R^d; \R^d)$.
Let $\Omega$ be a bounded $C^{1, 1}$ domain. Assume that  $|\B|$ does not vanish to infinite order
at any point on $\partial \Omega$.
Let $\kappa_0$ be defined by \eqref{kappa-1} and
\begin{equation}\label{W-3}
\aligned
\Gamma_0  & = \left\{ y \in \partial \Omega: \ \kappa (y) = \kappa_0 \right\}.
\endaligned
\end{equation}
Suppose that there exists $c>0$ such that \eqref{cond-1g}-\eqref{cond-2g}
hold for any $y\in \Gamma_0$.
Then
\begin{equation}\label{up-dn-5}
\lambda^{D\!N} (\beta \A, \Omega)
\le \beta^{\frac{1}{\kappa_0 +2}} \Theta_{D\!N}
+ C \beta^{\frac{1}{\kappa_0+4}}
\end{equation}
for $\beta>1$, where $\Theta_{D\!N} $ is given by \eqref{co-2}.
\end{thm}

\begin{proof}

This follows directly from Lemma \ref{lemma-up2-4}.
\end{proof}

\begin{thm}\label{thm-up2c}
Suppose that $\A\in C^\infty(\R^d; \R^d)$.
Let $\Omega$ be a bounded $C^1 $ domain. Assume that  $|\B|$ does not vanish to infinite order
at any point on $\partial \Omega$.
Then
\begin{equation}\label{up2-c}
\aligned
  \limsup_{\beta\to \infty} 
\beta^{-\frac{1}{\kappa_0 +2}}
\lambda^{D\!N} (\beta \A, \Omega)
\le \Theta_{D\!N}.
\endaligned
\end{equation}
\end{thm}

\begin{proof}

This follows from the proof of Lemma \ref{lemma-up2-4}.
Let $y\in \Gamma_0$.
Without the loss of generality, we may assume $y=0$, $n(y)=(0, \dots, 0, -1)$ and
\eqref{o-l} holds.
Since $\Omega$ is $C^1$,  we have  $|\nabla \phi_\beta(x^\prime)|= o(R \beta^{-\frac{1}{\kappa_0+2}})$.
Thus, by  Theorem \ref{thm-p7}, 
$$
\mu^{D\!N}(\A_0, \Omega_\beta\cap Q(0, 2R), \Omega_\beta) 
\le \lambda^{D\!N}(\A_0, \R^d_+) + R^{\kappa_0+1} o(R \beta^{-\frac{1}{\kappa_0+2}}) +C R^{-2}.
$$
This, together with \eqref{loc-101}, gives
$$
\aligned
\lambda^{D\!N} (\beta\A, \Omega)
&\le \mu^{D\!N}(\beta\A, Q(0, r) \cap \Omega, \Omega)\\
&\le \beta^{\frac{1}{\kappa_0+2}}
\left\{ \lambda^{D\!N} (\A_0, \R^d_+)
+ C R^{\kappa_0+1} o(R \beta^{-\frac{1}{\kappa_0+2}}) 
+ CR^{-2} \right\}
  \left\{ 1+ C R^{\kappa_0+2} \beta^{-\frac{1}{\kappa_0+2}} \right\}\\
\endaligned
$$
for any $1< R< \beta ^{\frac{1}{(\kappa_0+2)^2}}$.
It follows that
$$
\limsup_{\beta\to \infty}
\beta^{- \frac{1}{\kappa_0+2}}
\lambda^{D\!N} (\beta \A, \Omega)
\le \lambda^{D\!N} (\A_0, \R^d_+) +C R^{-2}
$$
for any $R>1$. By letting $R\to \infty$, we obtain \eqref{up2-c}.
\end{proof}


\section{Asymptotic expansions}\label{section-ae}

In this section  we establish the asymptotic expansion formulas \eqref{a-1} and \eqref{aa-1}  for 
$\lambda^D(\beta\A, \Omega)$,
$\lambda^N (\beta\A, \Omega)$,
and $\lambda^{D\!N} (\beta\A, \Omega)$.
Throughout the section, unless indicated otherwise,  we assume that $\A\in C^\infty(\R^d; \R^d)$ and 
$\Omega$ is a bounded $C^{1, 1} $ domain. As before, 
for $\lambda^D(\beta\A, \Omega) $ and $\lambda^N (\beta \A, \Omega)$, we assume that 
 $|\B|$ does not vanish to infinite order at any point in $\overline{\Omega}$.
It follows that there exist $C_0, c_0>0$ such that 
\begin{equation}\label{ae-m}
c_0\le \sum_{|\alpha|\le \kappa_* }
|\partial^\alpha \B (x) |
\quad
\text{ and } \quad
\sum_{|\alpha|\le \kappa_*+1}
|\partial^\alpha \B(x)|\le C_0
\end{equation}
for any $x\in\overline{\Omega}$, where $\kappa_*$  is  defined by \eqref{kappa}.
By continuity, without the loss of generality,  we may assume that \eqref{ae-m} holds for any $x\in \widetilde{\Omega}$, where $\widetilde{\Omega}
=\{ x\in  \R^d: \text{\rm dist}(x, \Omega)< r_0 \}$.
Thus,  by Remark \ref{re-op-2},  $\beta \B$ satisfies the condition \eqref{c-1} in $\widetilde{\Omega}$ with $\eta (t)=Ct$.
Consequently, by Theorem \ref{thm-op1}, for $\beta > C r_0^{-2}$, 
\begin{equation}\label{ae-op}
c\int_\Omega
\{ m(x, \beta \B )\}^2 |\psi|^2
\le \int_\Omega |(D+\beta \A)\psi|^2
\end{equation}
for any $\psi  \in C^1(\overline{\Omega}; \C)$,
where $C, c>0$ depend only on $\kappa_*$, $\Omega$ and $(c_0, C_0)$ in \eqref{ae-m}.

\begin{lemma}\label{lemma-m}
Suppose  $\B$ satisfies \eqref{ae-m} for $x\in \widetilde{\Omega}$. 
Then, 
\begin{equation}\label{ae-m-2}
m(x, \beta\B) \ge c \bigg\{ \beta^{\frac{1}{\kappa_*+2}}
+  \beta^{\frac{1}{\kappa_*+1} } \bigg(  \sum_{|\alpha| \le \kappa_*-1}
|\partial^\alpha  \B(x)|\bigg)^{\frac{1}{\kappa_*+1}}
\bigg\}
\end{equation}
for any $x\in \overline{\Omega}$ and $\beta\ge C$, where
$C, c>0$ depend only on $\kappa_*$, $\Omega$,  and $(c_0, C_0)$ in \eqref{ae-m}.
\end{lemma}

\begin{proof}

We first show that $m(x, \beta\B)\ge c\, \beta^{\frac{1}{\kappa_*+2}}$.
Fix $x_0\in \overline{\Omega}$.
Let $\P_\ell $ denote the $\ell^{th}$ Taylor polynomial of $\B$ at $x_0$, 
where  $0\le \ell \le \kappa_*$. Then, for $0< t< c$, 
$$
\aligned
\max_{Q(x_0, t)} |\B|
 &\ge  \max_{Q(x_0, t)} |\P_{\kappa_*}   | - C t^{\kappa_* +1}
 \ge c \sum_{|\alpha| \le \kappa_* }
|\partial^\alpha \B (x_0)| t^{|\alpha|} - C t^{\kappa_*+1}\\
&\ge c\,  t^{\kappa_*} \sum_{|\alpha|\le \kappa_*} |\partial^\alpha \B (x_0)| - C t^{\kappa_* +1}\\
&\ge c \,  t^{\kappa_*}, 
\endaligned
$$
where we have used  \eqref{ae-m}.
It follows that 
$$
\max_{Q(x_0, t)} |\beta \B| \ge c \beta t^{\kappa_*} > t^{-2}, 
$$
if  $t> C \beta^{-\frac{1}{\kappa_*+2}}$ and $C>1 $ is large.
This implies that $\frac{1}{m(x_0, \beta\B)} \le  C \beta^{-\frac{1}{\kappa_*+2}}$ and hence, $m(x_0, \beta \B) \ge c\,  \beta^{\frac{1}{\kappa_*+2}}$.

Next,   we fix $\delta \in (0, 1)$ and suppose that
\begin{equation}\label{m-11}
\beta^{\frac{1}{\kappa_*+2}}
\le  \delta  \beta^{\frac{1}{\kappa_*+1}  }
\bigg( \sum_{|\alpha|= \ell }
|\partial^\alpha \B(x_0)| \bigg)^{\frac{1}{\kappa_*+1}}
\end{equation}
for some $0\le \ell \le \kappa_*-1$.
Note that 
$$
\aligned
\max_{Q(x_0, t)} |\B|
 &\ge  \max_{Q(x_0, t)} |\P_\ell  | - C t^{\ell +1} \ge c \sum_{|\alpha| \le \ell}
|\partial^\alpha \B (x_0)| t^{|\alpha|} - C t^{\ell+1}\\
&\ge c\,  t^\ell \sum_{|\alpha|= \ell} |\partial^\alpha \B (x_0)| - C t^{\ell +1}\\
&\ge c \,  t^\ell \sum_{|\alpha| =  \ell} |\partial^\alpha \B (x_0)|
\endaligned
$$
if  $0< t< c$ and $0< t < c\sum_{|\alpha|= \ell} |\partial^\alpha \B(x_0)|$.
It follows that
$$
\max_{Q(x_0, t)} |\beta \B|
\ge c \beta  t^\ell \sum_{|\alpha| = \ell} |\partial^\alpha \B (x_0)|
> t^{-2}
$$
if $t \in (0, c)$ satisfies the conditions that  $t < c\sum_{|\alpha|= \ell} |\partial^\alpha \B(x_0)|$ and
$$
 c \beta  t^{\ell+2}  \sum_{|\alpha|= \ell} |\partial^\alpha \B (x_0)|>1.
 $$

Finally, let 
$$
t_0=  T    \beta^{-\frac{1}{\kappa_*+1 }} 
 \bigg(\sum_{|\alpha|= \ell} 
|\partial^\alpha \B (x_0)| \bigg)^{-\frac{1}{\kappa_*+1}}.
$$
By \eqref{m-11}, we have  $t_0\le  C T   \beta^{-\frac{1}{\kappa_*+2}}< c$, if $\beta$ is large,  and
$$
c \beta  t_0^{\ell +2 }  \sum_{|\alpha|= \ell} |\partial^\alpha \B (x_0)|
\ge c\, T^{\ell +2}>1,
$$
if $ T $ is large.
Moreover, using \eqref{m-11}, one can verify that  $t_0< c \sum_{|\alpha|=\ell} |\partial^\alpha \B (x_0)|$ if  $\delta $ is small.
As a result, we obtain  $\frac{1}{m (x_0, \beta \B)} < t_0$ and hence,
$$
m(x_0, \beta \B)> t_0^{-1} \ge c \beta^{\frac{1}{\kappa_*+1}} \bigg(  \sum_{|\alpha|=\ell}
|\partial^\alpha \B(x_0)|\bigg)^{\frac{1}{\kappa_*+1}},
$$
which leads to \eqref{ae-m-2}.
\end{proof}

\begin{remark}\label{re-m-1}
{\rm

For $\lambda^{D\!N}(\beta\A, \Omega)$,
we shall assume that $|\B|$ does not vanish to infinite order at any point on $\partial\Omega$.
Let $\kappa_0$ be defined by \eqref{kappa-1}.
Then there exists $ c_1>0$ such that 
\begin{equation}\label{ae-m-1}
c_1\le \sum_{|\alpha|\le \kappa_0 }
|\partial^\alpha \B (x) |
\end{equation}
for any $x\in\partial\Omega$.
By continuity, without the loss of generality,  we may assume that  \eqref{ae-m-1} holds for any $x\in \R^d$ with dist$(x, \partial\Omega)< r_0$.
It follows from the proof of Lemma \ref{lemma-m} that 
\begin{equation}\label{ae-m-3}
m(x, \beta\B) \ge c
\bigg\{ \beta^{\frac{1}{\kappa_0+2}}
+ \beta^{\frac{1}{\kappa_0+1}}
\bigg( \sum_{|\alpha|\le \kappa_0-1}
| \partial^\alpha \B (x)| \bigg)^{\frac{1}{\kappa_0+1}}  \bigg\}
\end{equation}
for  $\beta>C$ and $x\in \R^d$ with dist$(x, \partial\Omega)< r_0/2$, where $C, c>0$  depend only on $\kappa_0$,
$\| \B \|_{C^{\kappa_0+1} (\overline{\Omega})}$  and $c_1$ in \eqref{ae-m-1}.
Also, note that by the proof of  Theorem \ref{thm-op1},
\begin{equation}\label{as-m-3a}
c\int_{  \{ x\in \Omega: \text{\rm dist}(x, \partial\Omega)< cr_0\} }
\{ m(x, \beta \B)\}^2 |\psi|^2
\le \int_\Omega | (D+\beta \A)\psi|^2
\end{equation}
for any $\psi \in C^1(\overline{\Omega}; \C)$.
}
\end{remark}

\begin{lemma}\label{lemma-as-a}
Let $0< r< cr_0$. 
Then there exists a  finite set  of balls  $\{ \b(z_\ell, t_\ell)\}$  with the  properties 
that  (1) either $z_\ell \in \partial\Omega$ or $\b(z_\ell, 2 t_\ell )\subset \Omega$;  (2) $r\le t_\ell \le 48r$; 
(3)  if $z_\ell \in \partial\Omega$, then either $z_\ell \in \Gamma_2$ or  $\b(z_\ell,2 t_\ell)  \cap \Gamma_2 = \emptyset$;
(4) if $z_\ell \in \Omega$, then either $z_\ell \in \Gamma_1$ or $\b(z_\ell, 2t_\ell )\cap \Gamma_1 =\emptyset$.
Moreover, 
\begin{equation}\label{ae-a1}
\aligned
 \lambda^D (\beta \A, \Omega) 
\ge  \min_{\ell}  \lambda^D (\beta \A, \b (z_\ell, t_\ell ) \cap \Omega) -\frac{C}{r^2}, 
\endaligned
\end{equation}
where  $C$ depends only on $d$ and $\Omega$.
\end{lemma}

\begin{proof}

Let $\{\varphi_\ell\}$ be a partition of unity such that  (1) $\sum_\ell \varphi_\ell^2  =1$ in $\overline{\Omega}$, (2)
either $x_\ell \in \partial \Omega$ or $\b(x_\ell, 8r)\subset \Omega$, (3) supp$(\varphi_\ell ) \subset \b(x_\ell, r_\ell)$
where $r_\ell =16r$ if $x_\ell \in \partial\Omega$ and $r_\ell =r$ if $x_\ell \in \Omega$,
(4) $|\nabla \varphi_\ell|\le C/r$, and
 (5) $\sum_\ell \chi_{\b(x_\ell, r_\ell )} \le C$.
It follows that 
\begin{equation}\label{ae-a3}
\lambda^D (\beta \A, \Omega)
\ge \min_\ell \lambda^D (\beta\A, \b (x_\ell, r_\ell )\cap \Omega) -\frac{C}{r^2}.
\end{equation}
For each $\ell$, the ball $\b(z_\ell, t_\ell)$ is selected as follows.
If  $x_\ell \in \partial\Omega$ and  $\b(x_\ell, 2r_\ell) \cap \Gamma_2 =\emptyset $, we let $z_\ell =x_\ell$ and  $t_\ell =r_\ell =16r$.
If $x_\ell \in \partial\Omega$ and $\b(x_\ell,2 r_\ell ) \cap \Gamma_2 \neq \emptyset$,
we choose $z_\ell \in \Gamma_2 $ such that $|x_\ell -z_\ell |<2 r_\ell $.
Then $\b (x_\ell, r_\ell ) \subset \b (z_\ell, 3r_\ell)$ and hence, 
\begin{equation}\label{ae-a2}
\lambda^D (\beta \A, \b(x_\ell, r_\ell ) \cap \Omega) \ge 
\lambda^D (\beta \A, \b(z_\ell, t_\ell) \cap \Omega),
\end{equation}
where $t_\ell =3 r_\ell =48r$.
If $x_\ell \in \Omega$ and $\b(x_\ell ,2 r_\ell) \cap \Gamma_1=\emptyset$, we let $z_\ell =x_\ell$ and $t_\ell=r_\ell=r$.
Finally, if $x_\ell \in \Omega$ and $\b(x_\ell, 2r_\ell )\cap \Gamma_1 \neq \emptyset$, we choose $z_\ell \in \Gamma_1$
such that $|z_\ell -x_\ell |<2 r_\ell$ and let $t_\ell =3 r_\ell =3r$.
As a result, $\b(z_\ell,  2t_\ell)\subset \b (x_\ell, 8r) \subset \Omega$ and the inequality \eqref{ae-a2} continues to hold.
This, together with \eqref{ae-a3}, gives \eqref{ae-a1}.
\end{proof}

\begin{remark}\label{re-ae-1}
{\rm
Let $\{ \b(z_\ell, t_\ell) \}$ be the  finite  set of balls constructed in the  proof of Lemma \ref{lemma-as-a}.
A  similar argument yields 
\begin{equation}\label{ae-an1}
\aligned
\lambda^N (\beta \A, \Omega)
 & \ge \min \left\{
 \min_{z_\ell \in \partial \Omega}
  \mu^N (\beta \A, \b (z_\ell, t_\ell ) \cap \Omega, \Omega), 
  \min_{z_\ell \in \Omega}
  \lambda^D (\beta \A, \b (z_\ell, t_\ell) )
\right\}-\frac{C}{r^2},\\
\endaligned
\end{equation}
using the monotonicity of $\mu^N(\beta \A, \mathcal{O}, \Omega)$  in $\mathcal{O}$.
}
\end{remark}

\begin{thm}\label{main-thm-gL}
Let $\Omega$ be a bounded $C^{1, 1}$ domain in $\R^d$.
Suppose that  $|\B|$ does not vanish to infinite order at any point in $\overline{\Omega}$.
Assume \eqref{cond-1g} holds  for any $y\in \Gamma_*$ and \eqref{cond-2g} holds for any $y\in \Gamma_2$.
Further assume that  there exists $c>0$ such that 
\begin{equation}\label{cond-4g}
\sum_{|\alpha|\le \kappa_*-1}  |\partial^\alpha \B (x)|
 \ge c \, \text{\rm dist}(x, \Gamma_*)
 \qquad
 \text{ for any } x\in \overline{\Omega},
\end{equation}
and that 
\begin{equation}\label{cond-5g}
\text{\rm dist} (y, \partial\Omega) \ge c\,  \text{\rm dist} (y, \Gamma_2) \quad \text{ for any } y\in \Gamma_1.
\end{equation}
Then
\begin{equation}\label{ae-2g}
\lambda^D(\beta \A, \Omega)
\ge  \Theta_D \beta^{\frac{2}{\kappa_*+2}} - C \beta^{\frac{1}{\kappa_*+2} +\frac{1}{\kappa_*+4}}, 
\end{equation}
\begin{equation}\label{ae-3g}
\lambda^N(\beta \A, \Omega)
\ge  \Theta_N \beta^{\frac{2}{\kappa_*+2}} -C \beta^{\frac{1}{\kappa_*+2} +\frac{1}{\kappa_*+4}}, 
\end{equation}
for  $\beta > 1$,
where  $\Theta_D, \Theta_N $ are given by \eqref{co-1}. 
\end{thm}

\begin{proof}
Assuming  $\beta$ is sufficiently large, 
we give the proof for \eqref{ae-2g}.
A  similar argument, using  Remark \ref{re-ae-1} and Theorem \ref{thm-local-3},
 yields   \eqref{ae-3g}.

First, we use Lemma \ref{lemma-as-a} with 
\begin{equation}\label{r1}
r=\beta^{-\frac{1}{2( \kappa_* +2)} -\frac{1}{2(\kappa_*+4)}}
\end{equation}
   to obtain 
\begin{equation}\label{d-1-1}
\lambda^D (\beta \A, \Omega)
\ge \min_\ell  \lambda^D (\beta \A, \b(z_\ell, t_\ell)\cap \Omega) - C \beta^{\frac{1}{\kappa_* +2} +\frac{1}{\kappa_*+4}}.
\end{equation}
We will show that for each $\ell$,
\begin{equation}\label{d-1-2a}
\lambda^D (\beta \A,  \b(z_\ell, t_\ell ) \cap \Omega)
\ge \Theta_D \beta ^{\frac{2}{\kappa_* +2}} - C \beta^{\frac{1}{\kappa_* +2} +\frac{1}{\kappa_*+4}}
\end{equation}
for $\beta$ large.

Next, observe that if $z_\ell\in \Gamma_*$, the inequality \eqref{d-1-2a} follows readily from Theorems \ref{thm-local-1}
and \ref{thm-local-2}. Here we have used the fact that either $z_\ell \in \partial\Omega$ or $\b(z_\ell, 2 t_\ell)\subset \Omega$
as well as  the conditions \eqref{cond-1g}-\eqref{cond-2g}, which ensure the constants $C$ in \eqref{AE-I} and \eqref{AE-D} are uniform
in $y\in \Gamma_*$.
We now consider the case $z_\ell \in  \Omega$ and $\b(z_\ell, 2t_\ell) \cap \Gamma_1 =\emptyset$.
Since   $\b(z_\ell, 2 t_\ell)\subset \Omega$, we have $\b(z_\ell, 2t_\ell)\cap \Gamma_*=\emptyset$.
It follows by \eqref{ae-op} that  
$$
\aligned
\lambda^D (\beta \A,  \b(z_\ell, t_\ell ) )
 & \ge c \inf_{x\in \b(z_\ell, t_\ell)}  \{  m (x, \beta \B) \}^2\\
 &\ge c  \inf_{x\in \b(z_\ell, t_\ell)}  \beta^{\frac{2}{\kappa_*+1}}
 \bigg(  \sum_{|\alpha|\le \kappa_*-1} | \partial^\alpha \B (x)| \bigg)^{\frac{2}{\kappa_*+1}}\\
 &\ge c \inf_{x\in \b(z_\ell, t_\ell)}
 \beta^{\frac{2}{\kappa_*+1}}
 \big[ \text{\rm dist} (x, \Gamma_*)\big]^{\frac{2}{\kappa_*+1}}, 
\endaligned
$$
where we have used  Lemma \ref{lemma-m} and the assumption  \eqref{cond-4g}  for the second and third inequalities.
Since $\b(z_\ell, 2 t_\ell)\cap \Gamma_*)=\emptyset$,
we have dist$(x, \Gamma_*) \ge c\,  r$ for any $x\in \b(z_\ell,  t_\ell)$.
Thus, 
$$
\lambda^D (\beta \A,  \b(z_\ell, t_\ell ) )
\ge c\, \beta^{\frac{2}{\kappa_*+1}}
\beta^{-\frac{2 (\kappa_* +3)}{(\kappa_*+1)(\kappa_*+2)(\kappa_*+4)}},
$$
where we have used  \eqref{r1}.
A computation shows that 
$$
\frac{2}{\kappa_*+1} -\frac{2(\kappa_*+3)}{(\kappa_* +1)(\kappa_*+2)(\kappa_*+4)}
>   \frac{2}{\kappa_*+2}.
$$
Consequently,  \eqref{d-1-2a} holds for large $\beta$,   if $z_\ell \in \Omega$ and $\b(z_\ell, 2t_\ell)\cap \Gamma_1=\emptyset$.

Finally, consider the case $z_\ell \in \partial\Omega$ and $\b(z_\ell, 2 t_\ell) \cap \Gamma_2 =\emptyset$.
If $\b(z_\ell, 2t_\ell)\cap \Gamma_*=\emptyset$, the desired estimate \eqref{d-1-2a} follows as in the previous case.
Suppose   $\b(z_\ell, 2 t_\ell) \cap \Gamma_1 \neq \emptyset$.
It follows from the condition \eqref{cond-5g} that dist$(z_\ell, \Gamma_2) \le C r$.
Let $y\in \Gamma_2$ such that $\b(z_\ell, t_\ell) \subset \b(y,Cr)$. Then,
$$
\aligned
\lambda^D (\beta\A, \b(z_\ell, t_\ell)\cap \Omega)
 &\ge \lambda^D (\beta\A, \b(y, C r)\cap \Omega)\\
 &\ge \Theta_D \beta^{\frac{2}{\kappa_* +2}} -C \beta^{\frac{1}{\kappa_*+2} +\frac{1}{\kappa_*+4}},
\endaligned
$$
where we have used Theorem \ref{thm-local-2} for the last inequality.
\end{proof}

The next theorem gives asymptotic expansions for $\lambda^D(\beta\A, \Omega)$ and $\lambda^N(\beta\A, \Omega)$.

\begin{thm}\label{main-thm-g1}
Let $\Omega$ be a bounded $C^{1, 1}$ domain in $\R^d$.
Suppose that $|\B|$ does not vanish to infinite order at any point in $\overline{\Omega}$.
Assume that  \eqref{cond-1g} holds  for any $y\in \Gamma_*$ and   \eqref{cond-2g} holds for any $y\in \Gamma_2$.
 Also assume that  \eqref{cond-4g} and  \eqref{cond-5g} hold.
 Further suppose  that for any $y_0\in \overline{\Gamma}_1\cap \partial\Omega$ and $0< r< cr_0$, there exists $y_r\in \Gamma_1$ such that 
 \eqref{cond-3g} holds.
 Then
 \begin{equation}\label{AE-DN1}
 \aligned
 \lambda^D (\beta \A, \Omega) & = \Theta_D \beta^{\frac{2}{\kappa_*+2}} + O(\beta^{\frac{1}{\kappa_*+2} +\frac{1}{\kappa_*+4}}),\\
 \lambda^N (\beta \A, \Omega) & = \Theta_N \beta^{\frac{2}{\kappa_*+2}} + O(\beta^{\frac{1}{\kappa_*+2} +\frac{1}{\kappa_*+4}}),
\endaligned
\end{equation}
for $\beta>1$.
\end{thm}

\begin{proof}

The upper bounds for $\lambda^D(\beta\A, \Omega)$ and $\lambda^N(\beta\A, \Omega)$ are given by  Theorem \ref{thm-mup1},
 while the lower bounds are given by  Theorem \ref{main-thm-gL}.
\end{proof}

The rest of this section is devoted to $\lambda^{D\!N} (\beta \A, \Omega)$.
Let
\begin{equation}\label{o-r}
\Omega_r =\left\{ x\in \Omega: \text{\rm dist}(x, \partial\Omega)< r \right\}
\end{equation}
for $0< r< r_0$. Recall that $\Gamma_0$ is defined by \eqref{ga-1}.

\begin{lemma}\label {lemma-ae-dn}
Let $0< r< cr_0$.
Then there exists a finite set of balls $\{\b(y_\ell, s_\ell)\}$ with the properties that 
(1) $y_\ell \in \partial\Omega$ and
 $ r \le s_\ell \le 48r$;
(2) either $y_\ell \in \Gamma_0$ or $\b(y_\ell, 2 s_\ell) \cap \Gamma_0 =\emptyset$.
Moreover, if there exist  $\sigma>0$ and $c_0>0$  such that 
\begin{equation}\label{as-dn1}
\aligned
c_0\, \beta^{2\sigma} \int_{\Omega_{8r}} 
|\psi|^2 \le \int_\Omega |(D+ \beta \A)\psi|^2
\endaligned
\end{equation}
for any $\psi \in C^1 (\overline{\Omega}; \C)$ and $r \beta^\sigma\ge 1$, then
\begin{equation}\label{as-dn-2}
\aligned
\lambda^{D\!N} (\beta \A,  \Omega)
 & \ge \left\{ 1-  C (r \beta^\sigma )^{-1} \right\}
\min_\ell \mu^{D\!N} (\beta \A, \b (y_\ell, s_\ell )\cap \Omega,  \Omega),\\
\endaligned
\end{equation}
where $C$ depends only on  $\Omega$ and $(c_0, \sigma)$ in \eqref{as-dn1}.
\end{lemma}

\begin{proof}

Let $\{\varphi_\ell \}$ be a partition of unity satisfying the same conditions as those in the proof of Lemma \ref{lemma-as-a}.
Using the identity 
$$
\text{Re}
\int_\Omega
(D+\beta \A) \psi \cdot 
\overline{ (D+\beta \A) (\varphi^2 \psi) }
=\int_\Omega |(D+\beta \A) (\varphi \psi)|^2
-\int_\Omega |\nabla \varphi |^2 |\psi|^2,
$$
 we obtain 
\begin{equation}\label{d-2-2a}
\aligned
 & \min_\ell \mu^{D\!N} (\beta \A, \b(x_\ell, r_\ell) \cap \Omega, \Omega)\int_{\partial\Omega} |\psi|^2\\
&\le \sum_\ell 
\mu^{D\!N} (\beta \A, \b(x_\ell, r_\ell) \cap \Omega, \Omega) \int_{\partial\Omega}
|\varphi_\ell \psi|^2\\
&\le \sum_\ell \int_{\Omega} |(D+\beta\A)(\psi \varphi_\ell)|^2\\
&=\text{Re}
\int_\Omega (D+\beta\A)\psi  \cdot \overline{(D+\beta\A) (\Phi  \psi )}
+\sum_\ell \int_\Omega  |\nabla \varphi_\ell|^2 |\psi|^2
\endaligned
\end{equation}
for any $\psi \in C^1(\overline{\Omega}, \C)$,
where $\Phi=\sum_\ell \varphi_\ell^2   $. We point out that 
 the minimum and sums in \eqref{d-2-2a} and $\Phi$
    are   taken over only  those $\ell$'s for which $x_\ell \in \partial\Omega$.
 It is not hard to see that the right-hand side  of \eqref{d-2-2a} is bounded by
 \begin{equation}\label{d-2-2b}
 \aligned
 \int_{\Omega}  |(D+\beta\A)\psi|^2
 + \frac{C}{r}
 \left(\int_{\Omega} |(D+\beta\A)\psi|^2 \right)^{1/2} \left(\int_{\Omega_{8r}} |\psi|^2 \right)^{1/2}
 + \frac{C}{r^2}
 \int_{\Omega_{8r} } |\psi|^2.
 \endaligned
 \end{equation}
Using \eqref{as-dn1}, we deduce that  \eqref{d-2-2b} is bounded by
 $$
 \left\{ 1+ C r^{-1} \beta^{-\sigma} \right\}
 \int_\Omega |(D+\beta \A)\psi|^2.
 $$
 As a result, we have proved that 
 $$
  \min_\ell \mu^{D\!N} (\beta \A, \b(x_\ell, r_\ell) \cap \Omega, \Omega)\int_{\partial\Omega} |\psi|^2
  \le 
 \left\{ 1+ C r^{-1} \beta^{-\sigma } \right\}
 \int_\Omega |(D+ \beta \A)\psi|^2
$$
for any $\psi \in C^1(\overline{\Omega}; \C)$.
 This implies that
\begin{equation}\label{d-2-4}
\lambda^{D\!N} (\beta \A, \Omega)
\ge 
 \left\{ 1- C r^{-1} \beta^{-\sigma} \right\}\min_\ell \mu^{D\!N} (\beta \A, \b(x_\ell, r_\ell) \cap \Omega, \Omega).
 \end{equation}
Finally, if $\b(x_\ell, 2 r_\ell) \cap \Gamma_0 =\emptyset$, we let $y_\ell=x_\ell$ and $s_\ell =r_\ell$.
If $\b(x_\ell, 2 r_\ell) \cap \Gamma_0 \neq \emptyset$,
choose $y_\ell \in \b(x_\ell, 2 r_\ell)$ and let $s_\ell= 3r_\ell$.
Since $\b (x_\ell, r_\ell)\subset \b(y_z, 3r_\ell)$, we have
$$
\mu^{D\!N} (\beta\A, \b(x_\ell, r_\ell) \cap \Omega, \Omega)
\ge  \mu^{D\!N} (\beta\A, \b(y_\ell, s_\ell) \cap \Omega, \Omega),
$$
which gives \eqref{as-dn-2}.
\end{proof}

\begin{thm}\label{main-thm-gL1}
Let $\Omega$ be a bounded $C^{1, 1}$ domain in $\R^d$.
Suppose that $|\B|$ does not vanish to infinite order at any point on $\partial\Omega$.
Assume \eqref{cond-1g}  and \eqref{cond-2g} hold for any $y\in \Gamma_0$.
Further assume that  there exists $c>0$ such that 
\begin{equation}\label{cond-4g1}
\sum_{|\alpha|\le \kappa_0-1}  |\partial^\alpha \B (x)|
 \ge c \, \text{\rm dist}(x, \Gamma_{0})
 \qquad
 \text{ for any } x\in \partial\Omega.
\end{equation}
Then
\begin{equation}\label{ae-2g1}
 - C \beta^{\frac{\kappa_0+3}{(\kappa_0+2)(\kappa_0+4)}}
  \le 
  \lambda^{D\!N} (\beta \A, \Omega)
  - \Theta_{D\!N} \beta^{\frac{1}{\kappa_0+2}}
  \le  C \beta^{\frac{1}{\kappa_0+4}}
\end{equation}
for  $\beta > 1$,
where $\Theta_{D\!N} $ is given by \eqref{co-2}. 
\end{thm}

\begin{proof}

The upper bound for $\lambda^{D\!N} (\beta \A, \Omega) -\Theta_{D\!N} \beta^{\frac{1}{\kappa_0+2}}$ in \eqref{ae-2g1} is given by Theorem \ref{thm-up5a}
under the assumptions \eqref{cond-1g}-\eqref{cond-2g} for any $y\in \Gamma_0$.
To establish  the lower bound, 
we apply Lemma  \ref{lemma-ae-dn}  with 
\begin{equation}\label{r-2}
r=\beta^{-\frac{\kappa_0+3}{(\kappa_0+2)(\kappa_0+4)}}.
\end{equation}
To this end, 
we first  use Remark  \ref{re-m-1} to obtain 
 $$
 c\, \beta^{\frac{2}{\kappa_0+2}} \int_{\Omega_{cr_0}} |\psi|^2
 \le \int_\Omega |(D+\beta\A)\psi|^2
 $$
 for any $\psi \in C^1(\overline{\Omega}; \C)$.
 Thus, by Lemma \ref{lemma-ae-dn}, 
 \begin{equation}\label{d-2-10}
 \lambda^{D\!N} (\beta \A, \Omega)
 \ge \left\{ 1-C (r \beta^{\frac{1}{\kappa_0+2}})^{-1} \right\}
 \min_\ell \mu^{D\!N}(\beta \A, \b(y_\ell, s_\ell)\cap \Omega, \Omega).
 \end{equation}
 Assume $\beta$ is sufficiently large. 
We will show that  for each $\ell$, 
\begin{equation}\label{d-2-2}
\mu^{D\!N} (\beta \A,  \b(y_\ell, s_\ell  ) \cap \Omega, \Omega)
\ge \Theta_{D\!N} \beta ^{\frac{1}{\kappa_0 +2}} - C \beta^{\frac{1}{\kappa_0+4}},
\end{equation}
which, together with \eqref{d-2-10}, gives the first inequality in  \eqref{ae-2g1}.

We consider two cases.
Suppose $\b(y_\ell, 2 s_\ell)\cap \Gamma_{0}=\emptyset$.
It follows  by \eqref{op5-1} that 
$$
\aligned
\mu^{D\!N}  (\beta \A,  \b(y_\ell, s_\ell ) \cap \Omega, \Omega)
 & \ge c \inf_{x\in \b(y_\ell, s_\ell)\cap \partial \Omega}    m (x, \beta \B) \\
 &\ge c \beta^{\frac{1}{\kappa_0+1}}
   \inf_{x\in \b(y_\ell, s_\ell)\cap\partial \Omega}  \bigg( \sum_{|\alpha|\le \kappa_0-1} 
   | \partial^\alpha \B (x)|\bigg)^{\frac{1}{\kappa_0+1}}\\
 &\ge c \beta^{\frac{1}{\kappa_0+1}} \inf_{x\in \b(y_\ell, s_\ell)\cap\partial\Omega}
 \big[ \text{\rm dist} (x, \Gamma_0) \big]^{\frac{1}{\kappa_0+1}}\\
 &\ge  c  (\beta r )^{\frac{1}{\kappa_0+1}}\\
 &=c \beta^{\frac{\kappa_0^2 + 5\kappa_0 +5}{(\kappa_0+1)(\kappa_0+2)(\kappa_0+4)}}, 
  \endaligned
$$
where we have used Remark \ref{re-m-1} and \eqref{cond-4g1} for the second and third  inequalities.
Since
$$
{\frac{\kappa_0^2 + 5\kappa_0 +5}{(\kappa_0+1)(\kappa_0+2)(\kappa_0+4)}}
> \frac{1}{\kappa_0+2},
$$
we obtain \eqref{d-2-2} for the case $\b(y_\ell, 2s_\ell)\cap \Gamma_0=\emptyset$.

Finally, suppose $\b(y_\ell, 2s_\ell)\cap \Gamma_0  \neq \emptyset$. Then $y_\ell \in \Gamma_0$.
The inequality \eqref{d-2-2} follows readily from \eqref{AE-DN}.
\end{proof}


\section{Examples}

\subsection{The non-vanishing case}\label{sub-nv}

\begin{lemma}\label{lemma-const}
Suppose $\B$ is constant. Then
\begin{equation}
\lambda(\A, \R^d)=\lambda^D (\A, \mathbb{H}_n)=\text{\rm Tr}^+(\B)
\end{equation}
for any $n \in \mathbb{S}^{d-1}$, where $\text{\rm Tr}^+ (\B)= (1/2)  \text{\rm Tr} ( [ \B^* \B ]^{1/2})$.
\end{lemma}

\begin{proof}

It is known that $\lambda(\A, \R^d) = \text{\rm Tr}^+(\B)$.
See e.g. \cite{Helffer-book}.
To show $ \lambda^D (\A, \mathbb{H}_n)=\text{\rm Tr}^+(\B)$, 
by rotation, we may assume $\mathbb{H}_n =\R^d_+$.
In view of Lemma \ref{lemma-p1}, we have
$$
\aligned
\lambda(\A, \R^d)
& \le \lambda^D (\A, \R^d_+)
\le \lambda^D(\A, Q_+(0, R))\\
& =\lambda^D (\A, (0, -R/4) + Q_+ (0, R)) \\
&\le \lambda^D(\A, \b(0, R/4)) \\
&\le \lambda (\A, \R^d) + C R^{-2},
\endaligned
$$
where we have used the fact that $\B$ is constant for the equation.
By letting $R\to \infty$, we obtain $\lambda(\A, \R^d)=\lambda^D(\A, \R^d_+)$.
\end{proof}

 Assume that 
\begin{equation}\label{n-1}
\min_{x\in \overline{\Omega}} |\B(x)|>0.
\end{equation}
It follows  by Lemma \ref{lemma-const} that, 
\begin{equation}\label{const-1}
\Theta_D = \min_{y\in \overline{\Omega}} \text{\rm Tr}^+ (\B(y)),
\end{equation}
where we also use the fact that  $ \text{\rm Tr}^+ (\B(y))$ is continuous.

\begin{thm}\label{thm-n-1}
Suppose $\A\in C^2 (\R^d; \R^d)$ and \eqref{n-1} holds.
Let $\Omega$ be a bounded $C^{1, 1}$ domain. Then
\begin{equation}\label{n-1a}
\aligned
\lambda^D (\beta\A, \Omega)
 & =\beta  \min_{y\in \overline{\Omega}} \text{\rm Tr}^+ (\B(y)) +O(\beta^{\frac{3}{4}}), \\
\lambda^N (\beta\A, \Omega)
 &=\beta \Theta_N
+O(\beta^{\frac{3}{4}}),
\endaligned
\end{equation}
as $\beta\to \infty$, where $ \Theta_N$ is given by \eqref{co-1} with $\Gamma_1=\Omega$ and $\Gamma_2=\partial\Omega$.
\end{thm}

\begin{proof}
By the condition  \eqref{n-1}, we have  $\kappa_*=0$, $\Gamma_1 =\Omega$ and $\Gamma_2=\partial\Omega$. 
Moreover,  for any  $y\in \overline{\Omega}$,   $\P_{y} (x) = \B(y)$ is constant in $x$ and thus  its invariant subspace $V=\R^d$.
As a result, $\B$ satisfies the conditions in Theorem \ref{main-thm-g1}, from which \eqref{n-1a} follows.
\end{proof}

\begin{remark}
{\rm 
In the case of the Dirichlet condition,  under the assumption \eqref{n-1}, it was proved in \cite{Helffer-1996} that 
$$
\Theta_D \beta - C \beta^{\frac34}
\le \lambda^D(\beta \A, \Omega)
\le \Theta_D \beta + C \beta^{\frac23}, 
$$
which givers a better upper bound.
The case of the Neumann condition  was studied in \cite{Lu-1999, Helffer-2001} for $d=2$. 
The asymptotic expansion for $\lambda^N(\beta\A, \Omega)$ in \eqref{n-1a} was established  in \cite{Helffer-2001}  with 
$$
\Theta_N = \min \left( \inf_{\overline{\Omega}} |B_{12}|, \Theta \inf_{\partial\Omega} |B_{12}| \right),
$$
where $\Theta \in (0, 1)$ is a universal constant.
}
\end{remark}

\begin{thm}\label{thm-n-2}
Suppose $\A\in C^2 (\R^d; \R^d)$.
Let $\Omega$ be a bounded $C^{1, 1}$ domain in $\R^d$. Assume that 
\begin{equation}\label{n-2}
\min_{y\in \partial\Omega} |\B(y)|>0.
\end{equation}
Then
\begin{equation} \label{n2-b}
\Theta_{D\!N }\beta^{\frac12} - C \beta^{\frac38}
\le \lambda^{D\!N} (\beta\A, \Omega)
\le \Theta_{D\!N} \beta^{\frac12} + C \beta^{\frac14},
\end{equation}
for $\beta$ large, where $\Theta_{D\!N}$ is given by \eqref{co-2} with $\Gamma_0=\partial\Omega$.
\end{thm}

\begin{proof}
By \eqref{n-2},  we have 
 $\kappa_0=0$ and $\Gamma_0 =\partial\Omega$.
 The estimates in \eqref{n2-b} follow directly  from Theorem \ref{main-thm-gL1}.
 \end{proof}
 
 \begin{remark}
 {\rm
 In the case $d=2$, under the assumption \eqref{n-1}, it was proved in \cite{Helffer-2024} that
 $$
 \lambda^{D\!N} (\beta \A, \Omega)
 = \hat{\alpha} \inf_{x\in \partial\Omega} | B_{12} (x)|^{\frac12}  \beta^{\frac12}
 +o(\beta^{\frac12}),
 $$
 where $\hat{\alpha}\in (0, 1)$ is a universal constant.
 If $B_{12}$ is constant, a two-term asymptotic expansion  for $\lambda^{D\!N} (\beta \A, \Omega)$ is also obtained in \cite{Helffer-2024}.
 }
 \end{remark}
 

\subsection{The case of discrete wells}\label{sub-dis}

\begin{thm}\label{thm-d-1}
Let $\Omega$ be a bounded $C^{1, 1}$ domain in $\R^d$.
Suppose  that $|\B|$ does not vanish to infinite order at any points in $\overline{\Omega}$.
Also assume that the set $\Gamma_*=\Gamma_1\cup\Gamma_2$ in \eqref{gamma} is finite and  that  there exists $c>0$ such that 
\begin{equation}\label{dist-1}
\sum_{|\alpha|\le \kappa_*-1}  |\partial^\alpha \B (x)|
 \ge c \, \text{\rm dist}(x, \Gamma_*)
\end{equation}
for any $x\in \overline{\Omega} $.
Then
\begin{equation}\label{ae-2}
\lambda^D(\beta \A, \Omega)
= \Theta_D \beta^{\frac{2}{\kappa_*+2}} + O(\beta^{\frac{1}{\kappa_*+2} +\frac{1}{\kappa_*+4}}), 
\end{equation}
\begin{equation}\label{ae-2a}
\lambda^N(\beta \A, \Omega)
= \Theta_N \beta^{\frac{2}{\kappa_*+2}} + O(\beta^{\frac{1}{\kappa_*+2} +\frac{1}{\kappa_*+4}}), 
\end{equation}
as $\beta \to \infty$,
where  $\Theta_D, \Theta_N $ are given by \eqref{co-1}. 
\end{thm}

\begin{proof}

This follows from  the proof of Theorem  \ref{main-thm-g1}. Indeed, since $\Gamma_*$ is a finite, 
the estimates \eqref{AE-D} and \eqref{AE-N} hold uniformly for $y\in \Gamma_*$.
\end{proof}

\begin{remark}
{\rm 
Under the assumption that $\Gamma_*$ is finite, \eqref{ae-2} was proved in \cite{Helffer-1996}, which 
also established asymptotic expansions for all eigenvalues of $(D+\beta \A)^2$ subject to the Dirichlet condition.
}
\end{remark}

\begin{thm}\label{thm-d3}
Let $\Omega$ be a bounded $C^{1, 1}$ domain.
Suppose that  $|\B|$ does not vanish to infinite order at any point on $\partial\Omega$.
Also assume that the set $\Gamma_0$ in \eqref{ga-1} is finite and  that  there exists $c>0$ such that
\begin{equation}\label{dist-2}
\sum_{|\alpha|\le \kappa_0-1}  |\partial^\alpha \B (x)|
 \ge c \, \text{\rm dist}(x, \Gamma_0)
\end{equation}
for any $x\in \partial\Omega$.
Then
\begin{equation}\label{ae-3}
 \Theta_{D\!N}  \beta^{\frac{1}{\kappa_0+2}} -C \beta^{\frac{\kappa_0+3}{(\kappa_0+2)(\kappa_0+4)}}
 \le 
\lambda^{D\!N}(\beta \A, \Omega)
\le  \Theta_{D\!N}  \beta^{\frac{1}{\kappa_0+2}} + C\beta^{\frac{1}{\kappa_0+4} }
\end{equation}
for $\beta>1$,
where  $\Theta_{D\!N}  $ is  given by \eqref{co-2}. 
\end{thm}

\begin{proof}

This follows from the proof of Theorem \ref{main-thm-gL1}.
The assumption that $\Gamma_0$ is finite ensures that  the estimate \eqref{AE-DN}
holds uniformly for $y\in \Gamma_0$.
\end{proof}


\subsection{The first-order vanishing for $d=2$}\label{sub-order}

Let $d=2$ and  $B_{12} =\partial_1 A_2 -\partial_2 A_1$.
Suppose $\kappa_*=1$. It follows that  there exists $c_0>0$ such that 
\begin{equation}\label{d2-1}
|B_{12}(x) | +|\nabla B_{12} (x)|\ge c_0
\end{equation}
for any $x\in \overline{\Omega}$. 

\begin{thm}\label{thm-d2-1}
Suppose $\A\in C^2(\R^2; \R^2)$.
Let $\Omega$ be a bounded $C^{1, 1}$ domain in $\R^2$. Suppose   $\kappa_*=1$.
Also assume that if $y \in \partial\Omega$ and $B_{12}(y)=0$, then
\begin{equation}\label{d2-2}
|\langle \nabla B_{12}(y), T (y)\rangle|\ge c_0,
\end{equation}
where $T(y)$ is a unit tangent vector to $\partial\Omega$ at $y$.
Then
\begin{equation}\label{d2-2a}
\aligned
\lambda^D (\beta\A, \Omega) & =\Theta_D \beta^{\frac23} + O(\beta^{\frac{8}{15}}),\\
\lambda^N (\beta \A,  \Omega) & =\Theta_N \beta^{\frac23} + O(\beta^{\frac{8}{15}}),\\
\endaligned
\end{equation}
for $\beta$ large.
\end{thm}

\begin{proof}
Let $y\in \Gamma_* =\Gamma_1\cup \Gamma_2$.
Then $B_{12} (y)=0$.
By \eqref{d2-1}, we have $  |\nabla B_{12}(y)|\ge c_0$.
By the Implicit Function Theorem, it follows that $\Gamma_*$ is the union of a  finite number non-intersecting $C^1$ curves in $\overline{\Omega}$.
Moreover, 
\begin{equation}\label{d2-3}
|B_{12} (x)|\ge c \, \text{\rm dist}(x, \Gamma_*)
\end{equation}
for any $x\in \overline{\Omega}$.
Furthermore, the condition \eqref{d2-2} ensures that  if one of these curves intersects with $\partial\Omega$, 
they must  interest at a non-zero angle. In particular, $\Gamma_2$ is a finite set.

Next, note that  the first-order Taylor polynomial for $ B_{12}(x+y)$ at $0$ is 
$P(x)=\langle \nabla B_{12} (y), x \rangle$.
Its invariant subspace  is given by 
$$
V_y=\{ x\in \R^d: \langle x, \nabla B_{12}(y )\rangle =0  \}.
$$
Moreover,
$$
\sigma (y)=\min_{{\substack{ x\in V^\perp\\ |x|=1}}}
|P(x)|^{\frac12}= |\nabla B_{12} (y)|^{\frac12}\ge c_0^{\frac12}>0.
$$
Hence, $\B$ satisfies the assumptions in Theorem \ref{main-thm-g1}.
Consequently, \eqref{d2-2a}  follows from \eqref{AE-DN1}.
\end{proof}

\begin{remark}
{\rm 
The asymptotic expansion for $\lambda^N(\beta \A, \Omega)$ in \eqref{d2-2a} was established in \cite{Miqueu-2018},
while the formula for $\lambda^N(\beta \A, \Omega)$ was obtained earlier
in \cite{Helffer-1996}
under the additional assumption $\Gamma_2=\emptyset$.
As mentioned in Introduction, \cite{Helffer-1996} also treated the case $d\ge 3$, assuming $\Gamma_2=\emptyset $ and
$\Gamma_1$ is a submanifold. 
}
\end{remark}

\begin{remark}\label{re-d2}
{\rm
The argument used in the proof of Theorem \ref{main-thm-g1} can be used to treat the case where $\kappa_*=1$ and  $\Gamma_*=\partial\Omega$; i.e., $B_{12}(y)=0$ and $\nabla B_{12} (y)\neq 0$ for $y\in \partial\Omega$.
Indeed, as  in the proof of Theorem \ref{thm-d2-1}, for $y\in \Gamma_*$, 
the invariant subspace for the first-order Taylor polynomial for $B_{12}(x+y )$ is given by
$V=\{ x\in \R^2: \langle x, \nabla B_{12}(y) \rangle =0 \}$.
It follows that \
 $\sigma (y)=|\nabla B_{12}(y)|^{1/2}\ge c_0>$ and  that $n(y)=\pm \nabla B_{12}(y)/ |\nabla B_{12}(y)|$.
Thus,  the estimates in \eqref{AE-D} and \eqref{AE-N}  hold uniformly for $y \in \partial\Omega$.
As a result, the  asymptotic expansions in \eqref{d2-2a} continue to hold in the case $\kappa_*=1$ and $\Gamma_*=\partial\Omega$.
}
\end{remark}

\begin{thm}
Suppose $\A\in C^2(\R^2; \R^2)$.
Let $\Omega$ be a bounded $C^{1, 1}$ domain in $\R^2$. Suppose  $\kappa_0=1$.
Also assume that  \eqref{d2-2} holds for any $y \in \Gamma_0$.
Then
\begin{equation}\label{d2-5d}
\aligned
\Theta_{D\!N} \beta^{\frac13} - C \beta^{\frac{4}{15}}
\le \lambda^{D\!N} (\beta\A, \Omega)
\le \Theta_{D\!N} \beta^{\frac13} + C \beta^{\frac{1}{5}}
 \endaligned
\end{equation}
for $\beta$ large.
\end{thm}

\begin{proof}
Recall that $\Gamma_0 =\{ y\in \partial\Omega: B_{12}(y)=0 \}$.
It follows from the condition \eqref{d2-2} that $\Gamma_0$ is a finite set.
Moreover, if $B_{12}(\gamma(t))=0 $ for some curve $\gamma (t)$ in $\R^2$ such that $\gamma (0)=y\in \Gamma_0$,
then $ \langle\nabla B_{12} (y), \gamma^\prime (0)  \rangle =0$.
In view of \eqref{d2-3}, $\gamma^\prime (0)$ is not tangential to $\partial\Omega$ at $y$.
As a result, we deduce  by the Implicit Function Theorem that
$$
|B_{12} (x)| \ge c\, \text{\rm dist}(x, \Gamma_0)
$$
for any $x\in \partial\Omega$. Consequently, the inequalities in \eqref{d2-5d} follows from Theorem \ref{main-thm-gL1}
with $\kappa_0=1$.
\end{proof}

\begin{remark}
{\rm
The  inequalities in \eqref{d2-5d} continue to hold in the case $\kappa_0=1$ and $\Gamma_0=\partial\Omega$.
}
\end{remark}



 \bibliographystyle{amsplain}
 
\bibliography{S2025-1.bbl}

\bigskip

\begin{flushleft}

Zhongwei Shen,
Department of Mathematics,
University of Kentucky,
Lexington, Kentucky 40506,
USA.
E-mail: zshen2@uky.edu
\end{flushleft}

\bigskip

\end{document}